\documentclass[11pt,letterpaper]{amsart}

\usepackage{amsmath,amscd}
\usepackage{setspace}
\usepackage{multicol}
\usepackage[lite]{amsrefs} 
\usepackage{amssymb} 
\usepackage{graphicx} 
\usepackage{euscript}
\usepackage{color}
\usepackage{array}
\usepackage{picture}
\usepackage{epic}
\usepackage{tikz}
 \usetikzlibrary{backgrounds,cd}
\usepackage{comment}
\usepackage{enumerate}
\usepackage{url}
\usepackage[colorlinks=true, linkcolor=black, citecolor=black, urlcolor=black]{hyperref}

\usepackage{bookmark}
\usepackage[margin=1in]{geometry}
\usepackage{caption}
\usepackage{MnSymbol}
\usepackage{enumitem}
\usepackage{fancyhdr} 
\usepackage{braket}
\usepackage{longtable}
\usepackage{xcolor}
\usepackage{rotating}
\usepackage[all]{xy}
\usepackage{comment}
\usepackage{pdflscape}
\usepackage{tikz}
\usepackage{tikz-cd}
\usepackage{soul,color}
\usepackage[colorinlistoftodos]{todonotes}

\usepackage{tabmacD}
\newcommand{\Tableau}[2][sY]{{\text{\tableau[#1]{#2}}}}

\newcommand*{\fullref}[1]{\hyperref[{#1}]{\ref*{#1}. \nameref*{#1}}}

\usepackage{wrapfig}
\definecolor{due}{RGB}{0,76,147}
\usepackage{mathrsfs}
\usepackage{mathtools}
\usepackage{amsthm}
\usepackage{cleveref}
\usepackage{extarrows}
\usepackage{amsfonts}

\theoremstyle{definition}
\newtheorem{defi}{Definition}[section]
\theoremstyle{plain}
\newtheorem{thm}[defi]{Theorem}

\newtheorem{introthm}{Theorem}[section]

\newtheorem{prop}[defi]{Proposition}
\newtheorem{cor}[defi]{Corollary}
\newtheorem{lemma}[defi]{Lemma}
\theoremstyle{remark}

\newtheorem*{conv}{Standing assumptions}
\newtheorem{ex}[defi]{Example}

\theoremstyle{definition}

\newtheorem*{ack}{Acknowledgement}

\numberwithin{equation}{section}

\makeatletter
\@namedef{subjclassname@2020}{%
  \textup{2020} Mathematics Subject Classification}
\makeatother

\begin{document}
	\title[Motivic Classes of Isotropic Degeneracy Loci and Symmetric Orbit Closures]{Motivic Classes of Isotropic Degeneracy Loci and Symmetric Orbit Closures}
	\author{Minyoung Jeon}
	  \address{Center for Complex Geometry, Institute for Basic Science (IBS), Daejeon 34126, Republic of Korea}
\email{\url{minyoungjeon@ibs.re.kr}}


\begin{abstract}

We provide explicit formulas for computing the motivic Chern and Hirzebruch classes of degeneracy loci, especially those coming from the symplectic and odd orthogonal Grassmannians.
The Chern-Schwartz-MacPherson classes, K-theory classes, and Cappell-Shaneson L-classes arise as specializations of the motivic Chern and Hirzebruch classes. Our results are inspired by, and partially extends, those of Anderson--Chen--Tarasca in the case of ordinary Grassmannian degeneracy loci to isotropic and odd orthogonal Grassmannians as well as maximal even orthogonal Grassmannians. As applications, we obtain the motivic Chern and Hirzebruch classes of orthogonal and symplectic orbit closures in flag varieties.
\end{abstract}
\maketitle 
\setcounter{tocdepth}{1}
\tableofcontents 
\setcounter{tocdepth}{3}
\section{Introduction}

Characteristic classes of singular varieties, initiated in the 1960s, provide higher analogues of classical geometric invariants, such as the Euler–Poincaré characteristic and the signature, extending these notions to singular varieties. Notably, the Chern-Schwartz-MacPherson (CSM) class, Baum-Fulton-MacPherson's Todd class, and Cappell--Shaneson's L-class have been extensively studied. These three well-known characteristic classes are unified via the motivic characteristic classes of Brasselet, Sch\"urmann, and Yokura \cite{BSY}. The motivic Chern classes and Hirzebruch classes are natural transformations defined on the relative Grothendieck group of algebraic varieties over a singular variety $X$ and are expressed by polynomials in a formal variable $y$. Especially, the motivic Hirzebruch class $T_y$ recovers the CSM class at $y=-1$, the K-theoretic Todd class at $y=0$ and the Cappell-Shaneson L-class at $y=1$.

In the present paper, we give formulas for the motivic Chern and Hirzebruch classes for isotropic (Grassmannian) degeneracy loci in type C and the odd orthogonal degeneracy loci in type B, as a part of generalization of  the type A result \cite{ACT} to other types. As specializations, these formulas recover their CSM classes and L-classes. Related work includes numerous studies on the computation of their classes in cohomology, K-theory, and algebraic cobordism, including (but not limited to) \cites{AF,HIMN,A,HM}. In type A, for the degeneracy locus defined by a single morphism between vector bundle, the CSM class was computed by Parusi\'nski--Pragacz \cite{PP95} and the motivic Chern and Hirzebruch classes of Grassmannians and vexillary degeneracy loci were determined by Anderson--Chen--Tarasca \cite{ACT}. In the other types, the formulas for the L-classes or CSM classes of such degeneracy loci have remained largely unexplored. In contrast, the CSM classes and motivic Chern classes have been thoroughly investigated for Schubert varieties and Schubert cells in flag varieties, for example \cite{AM16,AMSS22,AMSS23,AMSS24,Jones,PR22,Zhang18}.

Let $\lambda=(\lambda_1\geq\lambda_2\geq\cdots\lambda_s\geq 0)$ be a partition determined by the sequence $\mathbf{q}=( q_1> q_2>\cdots> q_a>0> q_{a+1}>\cdots> q_s)$ and a positive integer $p$, see \eqref{eqn:lamb}. We consider the isotropic (or odd orthogonal) degeneracy loci ${\Omega}_\lambda$ in a smooth algebraic variety $X$ associated to the partition $\lambda$. To formulate the motivic Hirzebruch classes of $\Omega_\lambda$, we first establish a relation between the motivic Hirzebruch classes of the degeneracy locus $\Omega_\lambda$ and its certain resolution of $\widetilde{\Omega}_\lambda$. Specifically, we express the motivic Hirzebruch classes of the resolution by using a universal operator acting on $[\Omega_\lambda]\cap T_y(X)$, with $[\Omega_\lambda]$ given by the theta polynomial for the class of $\Omega_\lambda\in A^*(X)$, as shown below. In the formulas, $R_i$ arises from the first Chern class of the dual tautological line bundle  appearing in the resolution construction; moreover, $T_y(R_i)$ is understood as the power series $Q_y(R_i)$ (see \S2.1).
\begin{introthm}\label{TheoremC}
Let $\phi:\widetilde{\Omega}_\lambda\rightarrow\Omega_\lambda$ be the resolution of $\Omega_\lambda$ and $\iota:{\Omega}_\lambda\hookrightarrow X$ be the inclusion. Suppose that $a$ is the maximal number where $q_i>0$ for all $i\leq a$. 
The class of the resolution $(\iota\phi)_*T_y(\widetilde{\Omega}_\lambda)$ is given by 
\begin{enumerate}
\item (isotropic case) 
\[
(\iota\phi)_*T_y\left(\widetilde{\Omega}_\lambda\right)=\dfrac{\prod_{j\leq \rho_i}T_y(R_i+R_{j})}{\prod_{j<i}T_y(R_i-R_j)}\Theta_\lambda^\rho(c(1),\ldots,c(s))\scalebox{1.2}{$\cap$}\; T_y(X)
\]
where $c(i)=c(V-F_{q_i}-U)/T_y(R_i\otimes(V-F_{q_i}-U))$ for $i=1,\ldots,s$, 
\item (orthogonal case)   
\[
(\iota\phi)_*T_y\left(\widetilde{\Omega}_\lambda\right)=\dfrac{1}{2^a}\dfrac{\prod_{j\leq \rho_i}T_y(R_i+R_{j})\prod_{i=1}^aT_y(R_i)}{\prod_{j<i}T_y(R_i-R_j)}\Theta_\lambda^\rho(d(1),\ldots,d(s))\scalebox{1.2}{$\cap$}\; T_y(X),
\]
where $d(i)=c(V-F_{q_i}-U-M)/T_y(R_i\otimes(V-F_{q_i}-U))$ for $i=1,\ldots,a$ and $d(i)=c(V-F_{q_i}-U)/T_y(R_i\otimes(V-F_{q_i}-U))$ for $i=a+1,\ldots,s$. 
\end{enumerate}
\end{introthm}
 For undefined notations used in Theorem \ref{TheoremC}, we refer readers to later sections \S\ref{sec3}, \S\ref{sec4} for the theta polynomials and \S\ref{sec2} for the universal operator. As for the type A cases, the corresponding class for $\Omega_\lambda$ is represented by the determinantal formulas \cite[Theorem 3.1]{ACT}.

As the fibers of $\phi$ are non-constant, we stratify $\Omega_\lambda$ into finitely many sub-loci $\Omega_{\lambda^+}$, indexed by admissible weakly increasing sequences $\mathbf{k}=(k_1,\ldots, k_s)$, to compare the motivic Hirzebruch classes of $\Omega_\lambda$ and $\widetilde{\Omega}_\lambda$. This comparison is enabled because, on the corresponding locally closed strata $\Omega_{\lambda^+}^\circ\subseteq\Omega_{\lambda^+}$, the resolution $\phi$ becomes Zariski locally trivial, since the constant intersection dimensions define well-defined vector bundles over these smooth strata; see the proof of Theorem \ref{thm3.10} as well as \cite[\S 5.1]{ACT}. Here, admissibility is governed by the following ranges and inequalities:
Let $\rho_i=\#\left\{j\;|\; q_j\geq 1-q_i \;\text{for}\; i>j\geq1\right\}$. We require that $i\leq k_i\leq \min\{n+1-q_i,n+1-p\}$, that $k_{a+1}-k_a\leq -q_{a+1}+\mathrm{max}\{0,q_a+q_{a+1}-1\}$, and for $i>a+1$, that the sequence satisfies
$k_i-k_{i-1}+(\rho_{k_{i-1}}-\rho_{k_i})\leq q_{i-1}-q_i.
$
  For each such $\mathbf{k}$, the partition 
 $\lambda^+=\lambda^+(p,\mathbf{q}, {\mathbf{k}})$ is defined by a specific piecewise formula for $\lambda_{k_i}^+$
(see \eqref{eqn:LamPlus}) and the subsequent filling rule in \S\ref{sec3.3}. Using this stratification of $\Omega_\lambda$, the class $(\iota\phi)_*T_y(\widetilde{\Omega}_\lambda)$ in Theorem \ref{TheoremC} can be written as follows.

\begin{introthm}\label{mainStra}
For a partition $\lambda=\lambda(p,\mathbf{q})$ under the same assumption as in Theorem \ref{TheoremC}, we have
\[
(\iota\phi)_*T_y\left(\widetilde{\Omega}_{\lambda}\right)=\sum_{\mathbf{k}}(-y)^{\overline{\mathbf{k}}}\iota_*T_y\left(\Omega_{\lambda^+}\right),
\]
where the sum is over all admissible weakly increasing sequences $\mathbf{k}=(0<k_1\leq k_2\leq \cdots\leq k_s)$ as above, $\lambda^+=\lambda^+(p,\mathbf{q},\mathbf{k})$ is defined as in \eqref{eqn:LamPlus} and $\overline{\mathbf{k}}=\sum_{i=1}^s(k_i-i)$. 
  \end{introthm}
  
A concrete instance of this theorem appears in Example \ref{ex3.11}.
In a similar fashion, the stratification in our theorem is almost identical to the vexillary degeneracy loci of type A, with the only differences in the way of parametrizing the strata of the loci, i.e., we use partitions $\lambda^+$ while they use triples $\tau^+=(\mathbf{k},\mathbf{p},\mathbf{q})$.

The motivic Hirzebruch class of $\Omega_\lambda$ can be computed by repeated application of Theorem \ref{mainStra} to the strata of $\Omega_\lambda$, the sub-strata and so on, using the inclusion-exclusion principle. As a result, one can express the motivic Hirzebruch class of $\Omega_\lambda$ in terms of classes $(\iota\phi)_*T_y(\widetilde{\Omega}_{\mu})$ of the isotropic (or odd orthogonal) degeneracy sub-loci $\Omega_\mu\subseteq\Omega_\lambda$. Once we express the motivic Hirzebruch class of $\Omega_\lambda$ in terms of such classes $(\iota\phi)_*T_y(\widetilde{\Omega}_\mu)$, we can apply Theorem \ref{TheoremC} to obtain a polynomial in $y$ with coefficients in the Chern classes of certain vector bundles. That is, we establish the following theorem:

   \begin{introthm}\label{mainCal}
With Theorems \ref{TheoremC} and Theorem \ref{mainStra} combined together, the motivic Hirzebruch class of $\Omega_\lambda$
 can be computed for every $\lambda$.
   \end{introthm}
    
Explicit calculations of the motivic Hirzebruch class of $\Omega_\lambda$ as explained above can be found in Example \ref{main_illu}.

  As applications of our results, we concern orbit closures of the orthogonal group $O_n$ for any $n$ and symplectic group $Sp_n$ for even $n$ acting on the complete flag variety $Fl_n$. These actions have finitely many orbits. We refer to these orbits as $K_n$-orbits where $K_n$ is one of the groups $O_n$ or $Sp_n$; they are called {\it orthogonal orbit closures} when $K_n=O_n$ and {\it symplectic orbit closures} when $K_n=Sp_n$. These orbits and orbit closures are uniquely parametrized by the set of involutions $I_n^{O_n}$ and the set of fixed-point-free involutions $I_n^{Sp_n}$ in the symmetric group $S_n$, respectively. Explicit definition of these orbit closures will be discussed in \S\ref{sec5}. The orthogonal and symplectic orbit closures belong to two of the three known families of type A symmetric varieties, the third coming from the action of $GL_p\times GL_q$ on $Fl_n$ \cite{RS90}.
  
The geometry of symmetric orbit closures has played an important role in representation theory \cite{Tanisaki} and they generalize Schubert varieties \cite{Wyser}. Accordingly, many of the questions for Schubert varieties (or degeneracy loci) can be raised for these $K_n$-orbit closures. In this regard, the Chern--Mather classes, as one of characteristic classes, of the symmetric orbit closure given by the action of $GL_p\times GL_q$ on $Fl_n$ will be discussed in the author's joint work \cite{GJL}. 
   
While classes of the symmetric varieties in cohomologies \cite{Brendan,Wyser,WY} and K-theory \cite{MP} are well-studied, the characteristic classes of $K_n$-orbit closures such as the CSM classes are much less known. So, we present the formulas for the motivic Hirzebruch classes of the orthogonal and symplectic orbit closures, especially those associated with {\it vexillary} involutions (defined in \S\ref{sec:vex}):

           \begin{introthm}
      Fix $K_n=O_n$ or $Sp_n$. Let $z\in I_n^{K_n}$ be vexillary with its associated partition $\lambda=\lambda^{K_n}(z)$. Let $X_z^{K_n}$ be the corresponding $K_n$-orbit closure and $\phi:\widetilde{X}_z^{K_n}\rightarrow X_z^{K_n}$ the desingularization. Let $\iota:X_z^{K_n}\hookrightarrow Fl_n$. Then the class $(\iota\phi)_*T_y(\widetilde{X}_z^{K_n})$ is isomorphic to
                  \begin{enumerate}
  \item (orthogonal, $K_n=O_n$)
\[
\prod_{j\leq i-1}\dfrac{T_y(R_i+R_{j})}{T_y(R_i-R_j)}\mathrm{Pf}_\lambda(d(1),\ldots,d(\ell))\;\scalebox{1.2}{$\cap$}\; T_y(Fl_n),
\]
where $d(i)=c(i)/T_y(i)$ if $\ell$ is even,
 and if $\ell$ is odd, then $d(i)=c(i)/T_y(i)$ for $1\leq i\leq \ell$, and $d(i)=1$ for $i=\ell+1$.
\item (symplectic, $K_n=Sp_n$) 
\[
\dfrac{1}{2^\ell}\prod_{j\leq i-1}\dfrac{T_y(R_i+R_{j})}{T_y(R_i-R_j)}\prod_{i=1}^\ell (T_y(R_i))^2\mathrm{Pf}_\lambda(d(1),\ldots,d(\ell))\;\scalebox{1.2}{$\cap$}\; T_y(Fl_n),
\]
where $d(i)=c(i)/T_y(i)$ if $\ell$ is even,
 and if $\ell$ is odd, then $d(i)=c(i)/T_y(i)$ for $1\leq i\leq \ell$, and $d(\ell+1)=1$.
 \end{enumerate}
     \end{introthm}
  Undefined notations in the above theorem will be defined later in \S\ref{sec5}.
  By setting $y=-1$, we can compute the CSM classes of $K_n$-orbit closures $X_y^{K_n}$ in terms of sub-orbit closures $X_z^{K_n}\subseteq X_y^{K_n}$, as in the case of the degeneracy loci.

  Our results and techniques are mainly inspired by those in \cite{ACT} of type A and \cite{AF,HIMN} which studied the Chern class formulas for the fundamental classes and K-theoretic formulas for the degeneracy loci. We adapt their approach to our settings with careful analysis to get the motivic classes of the corresponding isotropic degeneracy loci. Unlike fundamental classes discussed in \cite{AF} and ordinary Grassmannian case \cite{ACT}, computing the motivic classes of isotropic versions require more delicate arguments, since they involve symmetric or skew-symmetric forms on vector bundles. We stratify $\Omega_\lambda$ by the loci $\Omega_{\lambda^+}$ associated with certain partition $\lambda^+$, along which the resolution $\widetilde{\Omega}_\lambda\rightarrow \Omega_\lambda$ is locally trivial. We then establish the relation between the push-forward of the motivic class of $\widetilde{\Omega}_\lambda$ and the motivic classes of the strata $\Omega_{\lambda^+}\subseteq \Omega_\lambda$, after which the inclusion–exclusion principle yields the motivic classes of the isotropic degeneracy loci.

  When $X$ to be taken as the Lagrangian or the odd maximal orthogonal Grassmannians, the motivic class of its Schubert varieties is obtained as expressions involving the motivic class of the corresponding Grassmannians. The specialization $y=-1$ recovers the CSM classes, in agreement with the formulas of \cite{AM16}, see Examples \ref{ex3.12} and \ref{ex4.7}.

A limitation of the present work is that it addresses degeneracy loci for odd orthogonal and symplectic Grassmannians, and only the maximal case in type D. The treatment of non-maximal even orthogonal degeneracy loci, as well as extensions to vexillary degeneracy loci in types B, C, and D, would require additional arguments beyond those developed here and is left for future research.

\section{Preliminaries}\label{sec2}
This section provides a brief review of the motivic Chern and Hirzebruch classes from \cite{BSY} and necessary notions and results on vector bundles and raising operators for our later computation of the motivic classes of the degeneracy loci.

\subsection{The motivic Chern and Hirzebruch classes}
Let $X$ be an algebraic variety over a base field of characteristic zero and $\mathcal{V}$ the category of complex algebraic varieties. Let $G_0(X)$ be the Grothendieck group of coherent sheaves of $\mathcal{O}_X$-modules, and $A_*(X)$ the Chow group. We denote by $K_0(\mathcal{V}/X)$ the relative Grothendieck group of algebraic varieties over $X$. Let $y$ be a variable. On $K_0(\mathcal{V}/X)$ we have the {\it motivic Chern class transformation} 
\[
mC:K_0(\mathcal{V}/X)\rightarrow G_0(X)\otimes\mathbb{Z}[y]
\] and the motivic Hirzebruch class 
\[
T_y:K_0(\mathcal{V}/X)\rightarrow A_*(X)\otimes \mathbb{Q}[y],
\]
commuting with pushdown for proper maps. Furthermore, one may relate these motivic classes by the following commutative diagram
\[
\begin{tikzcd}[column sep=tiny]
&K_0(\mathcal{V}/X)\ar[dr,"T_y"]  \ar[dl, "mC"'] 
&
&[1.5em] \\
G_0(X)\otimes\mathbb{Z}[y]\ar[rr,"td_{(1+y)}"]
&
& A_*(X)\otimes \mathbb{Q}[y] 
\end{tikzcd}
\]
as $td_{(1+y)}\circ mC=T_y$ (\cite[\S 3]{BSY}), where $td_{(1+y)}(\mathcal{F})=\sum_{i\geq 0}td_i(\mathcal{F})\cdot (1+y)^{-i})$ with $td_{i}$ the degree $i$ component of the Todd class transformation $td_*$ that is linearly extended over $\mathbb{Z}[y]$. To be specific, the natural transformation $td_{(y+1)}:G_0(X)\rightarrow A_*(X)\otimes\mathbb{Q}[y,1/(1+y)]$ is the reformulation of the singular Riemann-Roch theorem of Baum--Fulton--MacPherson by Yokura \cite{Yok}.

For smooth and pure-dimensional $X$ with tangent bundle $\mathcal{T}_X$, these motivic Chern and Hirzebruch classes satisfy the normalization conditions: 
\begin{equation}\label{eqn2.1}
\displaystyle mC\left(id_X\right)=\sum_{i\geq 0}\left[\wedge^i\mathcal{T}_X^\vee\right]y^i=:\lambda_y\left(\mathcal{T}_X^\vee\right)\scalebox{1.2}{$\cap$}\left[\mathcal{O}_X\right]
\end{equation}
where $\lambda_y:G^0(X)\rightarrow G^0(X)\otimes \mathbb{Z}[y]$ denotes the total $\lambda$-class transformation on the Grothendieck $G^0(X)$ of coherent locally free sheaves on $X$ and $\cap[\mathcal{O}_X]:G^0(X)\rightarrow G_0(X)$ is deduced by $\otimes \mathcal{O}_X$; moreover,
\[
T_y\left(id_X\right)=\displaystyle\prod_{i=1}^{\mathrm{dim}\;X}Q_y\left(\alpha_i\right)\scalebox{1.2}{$\cap$}[X]=:T_y\left(\mathcal{T}_X\right)\scalebox{1.2}{$\cap$}\left[X\right]
\]
where $\alpha_i$ is the Chern roots of $\mathcal{T}_X$, and $Q_y(\alpha)$ is the normalized power series
\[
Q_y\left(\alpha\right):=\dfrac{\alpha(1+y)}{1-e^{-\alpha(1+y)}}-\alpha y 
\]
in $\mathbb{Q}[y][[\alpha]]$. In addition, the operators $\lambda_y$ and $T_y$ satisfy the properties $\lambda_y(a+b)=\lambda_y(a)\lambda_y(b)$ and $T_y(a+b)=T_y(a)T_y(b)$ for all $a$ and $b$.

The transformations $mC$ and $T_y$ are also well-defined for arbitrary $X$ admitting stratifications. Let $\{X_i\}_{i\in I}$ be a stratification by locally closed smooth subvarieties $X_i$. Then 
\[
mC\left(id_X\right):=\sum_{i\in I}mC\left(X_i\rightarrow X\right),\quad\text{and}\quad T_y\left(id_X\right):=\sum_{i\in I}T_y\left(X_i\rightarrow X\right).
\]

The following lemma (\cite[Lemma 2.4]{ACT}) is useful when we consider stratifications later in this paper. Let $\{X_k\}_{k\in K}$ be a stratification of $X$ by locally closed strata $X_k$. Suppose that $f:Y\rightarrow X$ is a proper morphism such that over each stratum $X_k$ the map $f$ is locally trivial with smooth fiber $F_k$. We further assume that there is a unique top-dimensional stratum $X_0$ in the stratification of $X$. 
\begin{lemma}\label{lem:fib}
Let $f_!$ be the $K$-theoretic push-forward from $f$. Then the push-forwards are given by 
\[
f_!mC\left(Y\rightarrow X\right)=\sum_{k\in K}d_kmC\left(\overline{X}_k\hookrightarrow X\right),\quad\text{and}\quad f_!T_y\left(Y\rightarrow X\right)=\sum_{k\in K}e_kT_y\left(\overline{X}_k\hookrightarrow X\right)
\]
where $d_k$ and $e_k$ are defined by
\[
d_k:=\left(\int_{F_k}\lambda_y(\mathcal{T}_{F_k}^\vee)\right)-\sum_j d_j,\quad\text{and}\quad e_k:=\left(\int_{F_k}T_y(F_k)\right)-\sum_j e_j
\]
whose sums are over the index $j$ such that $X_k$ are contained in the closure $\overline{X}_j$ of $X_j$.
\end{lemma}

\subsection{Vector bundles}

Let $E$ be a vector bundle on a smooth variety $X$ and $Z\hookrightarrow X$ the zero locus of a regular section $s$ of the vector bundle $E\rightarrow X$.
\begin{lemma}[{\cite[Lemma 2.2]{ACT}}]\label{lem2.1}
If $s:X\rightarrow E$ meets the zero section of $E$ transversally, then the motivic Chern and Hirzebruch classes of $Z$ are given by 
\begin{align*}
mC\left(Z\hookrightarrow X\right)=\dfrac{\lambda_{-1}(E^\vee)}{\lambda_y(E^\vee)} mC\left(id_Z\right),\quad{\text{and}}\quad T_y\left(Z\hookrightarrow X\right)=\dfrac{c_{top}(E)}{T_y(E)}T_y\left(id_X\right).\\
\end{align*}
\end{lemma}

Let $E$ be a vector bundle of rank $e$ and its Chern roots $\alpha_i$ for $1\leq i\leq e$. Let $R$ be a formal variable. Then we define 
\[
T_y\left(R\otimes {E}\right):=\prod_{i=1}^eQ_y\left(R+\alpha_i\right).
\]
By definition, for a smooth variety $X$, if ${E}$ is the tangent bundle of $X$ and $R=0$, the above equation equals the motivic Hirzebruch class $T_y(X)$ of $X$.

A {\it virtual vector bundle} is a formal difference $[E]-[F]$ of isomorphism classes of vector bundles $E$ and $F$, and two virtual vector bundles $E-F$ and $E'-F'$ can be regarded as identical if they are the same in the Grothendieck ring $K(X)$ of vector bundles over $X$. Let the ranks of vector bundles $E$ and $F$ be $r_1$ and $r_2$ respectively. Then the {\it virtual rank} of the virtual vector bundle $E-F$ is given by its $0$th Chern character $ch(E-F)_0=\mathrm{rk}(E-F)=r_1-r_2$.

 Additionally, for a virtual vector bundle $\mathscr{E}=E-F$ of rank ${m}$, and for a line bundle $L$, setting $y=-1$ in $T_y(\mathscr{E}\otimes L)=T_y(L\otimes(E-F))$ gives $c(\mathscr{E}\otimes L)$, which is computed by
\[
c_p(\mathscr{E}\otimes L)=\sum_{i=0}^p\binom{m-i}{p-i}c_i(\mathscr{E})c_1(L)^{p-i}.
\]
 Here, for any integer $m$ and a nonnegative integer $k$, the generalized binomial coefficient is defined by
\[
\displaystyle\binom{m}{k}=\dfrac{m(m-1)\cdots(m-k+1)}{k!}=(-1)^{k}\binom{-m+k-1}{k}
\] for $k\geq 1$ and $\binom{m}{0}=1$.
Alternatively, it can be expressed as
\begin{equation}\label{ACTlem4.2}
c(\mathscr{E}\otimes L)=\left(1+c_1(L)\right)^ec_{\frac{1}{1+c_1(L)}}(\mathscr{E}),
\end{equation}
where $c_t(\mathscr{E}):=\sum_{i\geq 0}c_i(\mathscr{E})t^i$. These formulas originally appeared for vector bundles in \cite[Example 3.2.2]{Ful}, and were later generalized to virtual vector bundles in \cite[Lemma 4.2]{ACT}.

\subsection{Raising operators}
We define {\it raising operators} $R_i$ for $1\leq i\leq n$ on sequences $s=(s_1,\ldots, s_n)$ of nonnegative integers by increasing the $i$th index by $1$, i.e.,
\[
R_i\left(s_1,\ldots,s_i,\ldots,s_n\right)=\left(s_1,\ldots, s_{i+1},\ldots,s_n\right),
\]
and these operators are commute pairwise.
We may also regard a monomial $R=\prod R_j^{m_j}$ as a {\it raising operator}.

Let $A$ be a commutative ring and have elements $c(i)_r$ for $1\leq i\leq n$ and a nonnegative integer $r$. Let $c_s$ denote $c(1)_{s_1}c(2)_{s_2}\cdots c(n)_{s_n}$ for $s=(s_1,\ldots,s_n)$.  Given a finite formal sum $\sum_{s\in \mathbb{Z}_{\geq0}^n} a_sc_{s}$, $a_s\in A$, the raising operator $R$ operates on the formal sum as
\[
R\left(\sum_{s\in \mathbb{Z}_{\geq0}^n }a_sc_s\right)=\sum_{s\in \mathbb{Z}_{\geq0}^n }a_sc_{R(s)}.
\]
That is, the raising operator applies to the index $s$.  To be specific, for each $i$, the raising operator $R_i$ increases the index of $c(i)$ by one as
\[
R_i\left(c(1)_{s_1}\cdots c(i)_{s_i}\cdots c(n)_{s_n}\right)=c(1)_{s_1}\cdots c(i)_{s_i+1}\cdots c(n)_{s_n}
\]
and these $R_i$ are extended linearly over $\mathbb{Q}$, and multiplicatively on monomials in elements $c(i)$. Additionally, the polynomial ring $A[R_i]_{1\leq i\leq n}$ acts on expressions in the following manner
\[
\left(\sum b_RR\right)\left(\sum a_s c_s\right)=\sum_t\left(\sum_{R(s)=t}b_Ra_s\right)c_t.
\]
This polynomial action can be extended to an action of power series ring $A[[R_i]]_{1\leq i\leq n}$ on the set of any long sequences but in that case, one must require the condition that $c(i)_0=1$ for all $i$ and $c(i)_0=1$ when $i>N$ for some $N$. 

In much of the literature (see, for example, \cite{AF,Tam11}), raising operators are denoted by $R_{ij}$ and act by increasing the $i$th entry and decreasing the $j$th entry of a sequence. Instead we work with the simpler operators $R_i$ introduced above, in line with the convention of \cite[\S 1.3.1]{ACT}, and use $T_{ij}$ to denote the operator usually written as $R_{ij}$, in order to avoid unnecessary confusion.

 \section{Isotropic Degeneracy Loci}\label{sec3}

\begin{conv}
Throughout \S\ref{sec3} and \S\ref{sec4}, \(X\) is assumed to be smooth and irreducible over a field of characteristic zero, unless stated otherwise. In particular, we identify operational Chow cohomology with Chow homology via cap product with \([X]\). The characteristic factors appearing in our main formulas are interpreted in \(A^*(X)\otimes \mathbb{Q}[y]\), while the resulting Hirzebruch classes are obtained by capping with \(T_y(X)\) and therefore lie in \(A_*(X)\otimes \mathbb{Q}[y]\). The relevant rank conditions and isotropic flags will be specified in each case.
\end{conv}

As in Lemma \ref{lem:fib} and Lemma \ref{lem2.1}, the motivic Chern and Hirzebruch classes share an equivalent formal structure, so that we may restrict our attention to the motivic Hirzebruch classes for simplicity. 
In this section, we establish formulas for the motivic Hirzebruch class of resolutions of the isotropic degeneracy loci and then use them to express the Hirzebruch classes of these loci as polynomials in $y$.

We begin by defining the isotropic degeneracy loci. Let $V$ be a vector bundle of rank $2n$ over $X$, equipped with a non-degenerate skew-symmetric bilinear form $\langle\cdot,\cdot\rangle:V\otimes V\rightarrow \mathcal{O}_X$. Given a subbundle $E$ of $V$ over $X$, we define its orthogonal complement $E^\perp\subset V$ with respect to the form as the subbundle whose fiber at $x\in X$ is given by 
$(E^\perp)_x = \{ v \in V_x \mid \langle v, e \rangle = 0 \ \text{for all } e \in E_x \}.$
A vector bundle $E$ is isotropic if the vector bundle vanishes on the form, or equivalently $E\subset E^\perp$. 

Let $\bar{a}:=-a$ denote the negative of $a$. We fix a complete flag
\begin{equation}\label{flagC}
F_\bullet:0\;\scalebox{1.2}{$\subset$}\; F_{n}\;\scalebox{1.2}{$\subset$}\; F_{n-1}\;\scalebox{1.2}{$\subset$}\; \cdots \;\scalebox{1.2}{$\subset$}\; F_1\;\scalebox{1.2}{$\subset$}\; F_{\bar{1}}\;\scalebox{1.2}{$\subset$}\;\cdots\;\scalebox{1.2}{$\subset$}\; F_{\bar{n}}=V
\end{equation}
of isotropic subbundles of $V$ over $X$. Here each $F_i$ has rank $n+1-i$ and $(F_{i+1})^\perp=F_{\bar{i}}$, $\mathrm{rk}(F_{\bar{i}})=n+i$ for all $i\geq 1$. We note that $F_1$ is a Lagrangian subbundle of $V$.

Fix $1< p\leq n$. We consider a sequence of integers
\[
\mathbf{q}=\left( n\geq q_1> q_2>\cdots> q_a>0> q_{a+1}>\cdots> q_s> -n\right)
\] 
whose length is $s\leq n-p+1$, satisfying that $q_1,\ldots, q_a, -q_{a+1},\ldots,-q_{s}$ are distinct. Let 
\begin{equation}\label{eqn:rho}
\rho_i=\#\left\{j\;|\; q_j\geq 1-q_i \;\text{for}\; i>j\geq1\right\}.
\end{equation} 
Let $U:=U_p\subset V$ be an isotropic subbundle of $V$ over $X$ whose rank is $n+1-p$ for $1\leq p\leq n$. Given a complete flag $F_\bullet$ as in \eqref{flagC}, we define the {\it isotropic degeneracy locus} by
\[
\Omega_\lambda^C:=\{x\in X\;|\;\mathrm{dim}(U \;\scalebox{1.2}{$\cap$}\; F_{q_i})|_x\geq i, i=1,\ldots, s\}.
\]
Here, the partition $\lambda:=\lambda(p,\mathbf{q})=(\lambda_1\geq\cdots\geq\lambda_s\geq0)$ of length $s$ is defined by
\begin{equation}\label{eqn:lamb}
\lambda_i=\begin{cases}
q_i+p-1& if\; i\leq a\\
q_i+p-1+i-\rho_i& if\; i>a
\end{cases}
\end{equation}
of positive integers such that $\lambda$ is $(p-1)$-strict. That is, $\lambda_i>\lambda_{i+1}$ whenever $\lambda_i>p-1$.
For $p=1$, we take 
$
n\geq q_1>\cdots>q_s>0\quad\text{and}\quad
\lambda_i=q_i\quad\text{for all $i$.}
$
 It is worthwhile to note that $\Omega_\lambda^C$ is the special case of \cite{AF} with $p_i=p$ for all $i$.

The following proposition confirms that our definition coincides with the loci studied in the existing literature \cite{HIMN}.

\begin{prop}
The conditions defining the loci discussed in \cite[Definition 5.5]{HIMN} corresponds to the ones for our locus $\Omega_\lambda^C$.
\end{prop}
\begin{proof}
The number 
\begin{align*}
\gamma_i:&=\#\{j\;|\;1\leq j<i,\;\lambda_i+\lambda_j>2(p-1)+i-j\}\\
&=\#\{j\;|\;1\leq j<i,\;\chi_i+\chi_j\geq0\}
\end{align*}
in \cite[Definition 5.1]{HIMN} is exactly the same as $\rho_i$ with $k=p-1$, $\chi_i=q_i-1$ if $q_i>0$ and $\chi_i=q_i$ if $q_i<0$. In addition, $\gamma_i=i-1$ if $q_i>0$.
\end{proof}

\subsection{Motivic classes}
This subsection deals with the fundamental computation of the motivic Chern and Hirzebruch class of (isotropic) Grassmannian bundles and the basic case of the corresponding degeneracy loci.

For a vector bundle $\mathcal{V}$ on $X$, the Grassmannian bundle $\pi:Gr(d,\mathcal{V})\rightarrow X$ is the set of rank $d$ bundles of $\mathcal{V}$. Let us consider the exact sequence
\[
0\rightarrow \mathcal{S}\rightarrow \mathcal{V}\rightarrow \mathcal{Q}\rightarrow 0
\]
of the tautological bundles over $Gr(d,\mathcal{V})$. ($\mathcal{V}$ should be understood as $\pi^*(V)$. We omit notation for such pullbacks.) Then by virtue of \cite[Lemma 2.3]{ACT}, we have 
\begin{equation}\label{Grass}
mC(id_{Gr(d,\mathcal{V})})=\lambda_y\left(((\mathcal{S}^\vee\otimes \mathcal{Q})^\vee\right)\cdot \pi^*mC(id_X),\;\text{and}\;T_y(id_{Gr(d,\mathcal{V})})=T_y((\mathcal{S}^\vee\otimes \mathcal{Q})\cdot \pi^*T_y(id_X)
\end{equation}
over $X$.

The following lemma treats the basic case where the motivic Hirzebruch class of the degeneracy loci is computed directly. Let $L$ be a line bundle and $E\subset V$ a subbundle of rank $e$ on $X$.
\begin{lemma}\label{lem3.3}
If the locus $\iota: Z\hookrightarrow X$ where a map $L\rightarrow V/E$ vanishes is smooth of the expected codimension $rk(V/E)$, then its motivic Hirzebruch class is
\[
\iota_*T_y(Z)=\dfrac{1}{T_y(V/E\otimes R)}c_{2n-e}(V-E-L)\scalebox{1.2}{$\cap$}\; T_y(X),
\]
where $R$ is the raising operator acting on $c(V-E-L)$.
\end{lemma}
\begin{proof}
By the hypothesis, the section of $L^\vee\otimes(V/E)$ is regular, so that $Z$ is a local complete intersection with normal bundle $N_Z\cong L^\vee\otimes(V/E)$. This implies  the identities $c_{2n-e}(L^\vee\otimes (V/E))=c_{2n-e}(V/E-L)$ and
\[ \left(c_1(L^\vee)\right)^kc_i(V/E-L)=c_{i+k}(V/E-L)
\]
for $k\geq0$ and $i\geq 2n-e=rk(N_Z)$. The first Chern class $c_1(L^\vee)$ here plays the role of the raising operator $R$, see \cite[pg. 3]{AF}. Hence, we get the statement via Lemma \ref{lem2.1}.
\end{proof}

Since $E \subset V$ is a subbundle, the quotient $V/E$ represents the same class as $V-E$ in $K_0(X)$. We therefore freely identify $V/E$ and $V-E$ when applying the Hirzebruch class $T_y$, and likewise for Chern classes.

\subsection{Resolution of singularities}\label{sec3.2}
We continue to use the notation introduced in the definition of $\Omega_\lambda^C$, namely, $\lambda, \mathbf{q}, s, p, a$ the vector bundles $V, U$ and the flag $F_\bullet$, as in the beginning of \S \ref{sec3}. 

The case that follows examines the motivic Hirzebruch class of the push-forward of the resolution $\widetilde{\Omega}_\lambda$ of the degeneracy loci. 
We first construct the tower of projective bundles
\begin{equation*}\label{eqn3.3}
X=:X_0\xleftarrow{\pi_1} X_1:=\mathbb{P}(U)\xleftarrow{\pi_2} X_2:=\mathbb{P}(U/D_1)\xleftarrow{\pi_3} \cdots\xleftarrow{\pi_s} X_{s}:=\mathbb{P}(U/D_{s-1})
\end{equation*}
where $D_i/D_{i-1}$ is the tautological line bundle on $X_{i}$ and $rk(D_i)=i$ and $D_0=0$. We simplify the notation by omitting the pull-backs from the natural projections $\pi_i$. That is, $U$ means $\pi_i^*U$ on $X_i$, and similarly $D_{i-1}$ means $\pi_i^*D_{i-1}$ on $X_i$. We further remark that, for the rest of the paper, such pullbacks of bundles will be omitted from the notation.

 The variety $X_i$ is, in fact, parametrized by the filtration
$
D_1\subset D_2\subset\cdots\subset D_i
$
of subbundles of $U$, where $\mathrm{rk}(D_j)=j$ for $j=1,\ldots,i$  (\cite[p.1800]{ACT}).
We denote by $\pi_i:X_i\rightarrow X_{i-1}$ be the projection, and $\pi=\pi_s\circ\cdots\circ \pi_1$ be the composition of $\pi_i$. 

On $X_i$, we define a locus
\[
Z_i=\{(x,D_1\;\scalebox{1.2}{$\subset$}\; D_2\;\scalebox{1.2}{$\subset$}\; \cdots\;\scalebox{1.2}{$\subset$}\; D_i)\;|\;(D_j\;\scalebox{1.2}{$\subseteq$}\; F_{q_j})|_x\;\text{for all }\; j\leq i\},
\]
with natural inclusion map $Z_i\hookrightarrow X_i.$ Since $D_j$ is subbundle of $U$, the condition for $Z_i$ can be regarded as 
$D_j\subseteq (F_{q_j}\cap U)|_x$ for all $j\leq i$.
Then the restriction map $\phi$ of $\pi$ sends $\widetilde{\Omega}_\lambda=Z_s$ birationally onto $\Omega_\lambda^C\subset X$ and the map $\widetilde{\Omega}_\lambda\rightarrow \Omega_\lambda^C$ resolves the singularities as in \cite[\S 2.3, p.1749 (Intro.)]{AF}, \cite[\S2.3]{AF} and \cite{KL}. 
It is worth noting that the following diagram commutes:
\begin{equation}\label{digm1}
\begin{tikzcd}
\widetilde{\Omega}_\lambda\arrow[r, hookrightarrow, "\iota"] \arrow[d, "\phi"'] & X_s \arrow[d, "\pi"] \\
\Omega_\lambda^C \arrow[r, hookrightarrow, ""] & X
\end{tikzcd}
\end{equation}

In addition, we may view that the locus $Z_1$ is the zero section of $\mathrm{Hom}(D_1,V/F_{q_1})$ on $X_1$, and $Z_i$ is the zero locus of the corresponding section of $\mathrm{Hom}(D_i/D_{i-1},D_{i-1}^\perp/F_{q_i})$ for $i\leq a$ and $\mathrm{Hom}(D_i/D_{i-1},D_{\rho_i}^\perp/F_{q_i})$ for $i>a$ on $\pi_{i}^{-1}(Z_{i-1})\subseteq X_i$ for $1<i\leq s$ (\cite[\S2.2]{AF}, \cite[Proof of Lemma 5.10]{IMN}).

Let $\ell$ be a positive integer and $\rho=(\rho_1,\ldots,\rho_\ell)$ a sequence of nonnegative integers with $\rho_i<i$. For a unimodal sequence $\rho=(\rho_1,\ldots,\rho_\ell)$, a partition $\lambda=(\lambda_1\geq\cdots\geq \lambda_\ell)$ is said to be $\rho${\it-strict} if the sequence $\mu_i=\lambda_i+\rho_i$ is non-increasing (\cite[p.9]{AF}). 

We remark that the sequence $\rho$ defined by \eqref{eqn:rho} is in fact unimodal and the partition $\lambda(p,\mathbf{q})$ is $\rho$-strict.

 Given symbols $c(1),\ldots,c(\ell)$ and a $\rho$-strict partition $\lambda$, the 
{\it theta-polynomial} is defined by
\[
\Theta_\lambda^\rho(c(1),\ldots,c(\ell))=\dfrac{\prod_{1\leq j<i\leq \ell}(1-T_{ij})}{\prod_{1\leq j\leq\rho_i<i\leq\ell}(1+T_{ij})}c(1)_{\lambda_1}\cdots c(\ell)_{\lambda_\ell},
\]
where $T_{ij}$ is defined as in the Preliminaries.
We note that $(1-T_{ij})/(1+T_{ij})=1+2\sum_{k>0}(-1)^kT_{ij}^k.$
In particular, in case of $\rho_i=i-1$, the theta polynomial is a Schur Pfaffian, and when it comes to $\rho=\emptyset$, $\Theta(\rho)_\lambda$ is a Schur determinant. See \cite[Appendix A.2.]{AF} for more details on the notations and theta-polynomials.

We now state one of our main results of this paper. In the statement, $R_i$ denotes the class on $X$ induced by $c_1((D_i/D_{i-1})^\vee)$. Moreover, the symbols $R_i,$ and $V-F_{q_i}-U$ in the following statement are understood as the corresponding classes or virtual bundles on $X$, obtained from the resolution construction. Recall also that $T_y(R_i)=Q_y(R_i)$.

\begin{thm}\label{mainC}
Let $\widetilde{\Omega}_\lambda$ be the resolution of $\Omega_\lambda^C$ and $\iota:\widetilde{\Omega}_\lambda\hookrightarrow X_s$ be the natural inclusion. 
Let $s\leq n+1-p$. The class of the resolution $(\pi\iota)_*T_y(\widetilde{\Omega}_\lambda)$ is given by 
\[
 \dfrac{\prod_{j\leq \rho_i}T_y(R_i+R_{j})}{\prod_{j<i}T_y(R_i-R_j)} \prod_{1\leq i\leq s}\dfrac{1}{T_y(R_i\otimes(V-F_{q_i}-U))}\Theta_\lambda^\rho(c(1),\ldots,c(s))\scalebox{1.2}{$\cap$}\; T_y(X).
\]\label{mainC1}
where $c(i)=c(V-F_{q_i}-U)$ for $i=1,\ldots,s$.
\end{thm}
 In the above theorem, $T_y(R_i+R_j)$ is expressed by the series $Q_y(R_i+R_j)$ and so does $T_y(R_i-R_j)$. 
The subsequent sections focus on the proof of Theorem \ref{mainC}.

\subsubsection{Proof of Theorem \ref{mainC}}
We start with the proof of Theorem \ref{mainC} \eqref{mainC1}. Before proceeding, we need the following lemma regarding the motivic Hirzebruch class for the projection $\pi_i$ on $X_i$.
\begin{lemma}\label{lem3.5}
Let $\pi_i:X_i\rightarrow X_{i-1}$ be the projections for $1\leq i\leq s$ as before. If $s<n+1-p$, then we have
\[
T_y(X_i)= T_y((D_i/D_{i-1})^\vee\otimes U/D_i))\scalebox{1.2}{$\cap$}\; \pi_i^*T_y(X_{i-1})
\]
for all $i$.
 If $s=n+1-p$, then 
$
T_y(X_i)= T_y((D_i/D_{i-1})^\vee\otimes U/D_i))\cap \pi_i^*T_y(X_{i-1})$ for $1\leq i\leq s-1$, and
\[
T_y(X_s)= \pi_s^*T_y(X_{s-1}).
\]
\end{lemma}

\begin{proof}
In the case of $s<n+1-p$, by the construction, $\pi_i:X_i\rightarrow X_{i-1}$ for $1\leq i\leq s$ is a projective bundle with its fiber satisfying the condition 
$D_{i-1}\subset D_i\subset U$ such that
\[
D_i/D_{i-1}\;\scalebox{1.2}{$\subset$}\; U/D_{i-1}.
\]
So, by Verdier-Riemann-Roch theorem and \eqref{Grass}, we have the first statement. 

As for the second statement, the case of $1\leq i\leq s-1$ follows directly by the first statement. The assumption $s=n+1-p$ implies that $U/D_{s-1}$ is of rank $1$, and thus $X_s=X_{s-1}$. 
So, the map $\pi_s:X_{s}\rightarrow X_{s-1}$ becomes a trivial $\mathbb{P}^0$-bundle with fiber a point.
Thus, we have
\begin{align*}
T_y(X_s)=T_y\left(0\right)\scalebox{1.2}{$\cap$}\; \pi_s^*T_y(X_{s-1})=1\scalebox{1.2}{$\cap$}\; \pi_s^*T_y(X_{s-1})=\pi_s^*T_y(X_{s-1}),
\end{align*}
since $(U/D_{s})=0$, so that $(D_s/D_{s-1})^\vee\otimes U/D_s=0$. 
\end{proof}

Having proved the lemma, we are in a position to prove the first part of the theorem.

\begin{proof}[Proof of Theorem \ref{mainC}]
We first consider the case of $s<n+1-p$. 
By virtue of Lemma \ref{lem3.5} and Lemma \ref{lem3.3}, we have 
\begin{equation}\label{eqn3.60}
\iota_*T_y(\widetilde{\Omega})=\displaystyle\prod_{i=1}^s\dfrac{T_y((D_i/D_{i-1})^\vee\otimes U/D_i)}{T_y((D_i/D_{i-1})^\vee\otimes D_{\rho_i}^\perp/F_{q_i})}c_{\widetilde{\lambda}_i}(D_{\rho_i}^\perp-F_{q_i}-D_i/D_{i-1})\scalebox{1.2}{$\cap$}\;\pi^*T_y(X)
\end{equation}
where $\widetilde{\lambda}_i=\lambda_i+n-p+1-i$. We note that pullbacks of bundles are omitted from the notation. To be specific, the partition $\widetilde{\lambda}$ is realized as
\[
\widetilde{\lambda}_i=\begin{cases}
n+q_i-i&\text{if}\;q_i>0\\
n+q_i-\rho_i&\text{if}\;q_i<0.
\end{cases}
\]
 
 When it comes to the case $s=n+1-p$, we know $X_s\cong X_{s-1}$, so that the locus $Z_s$ can be regarded the locus in $Z_{s-1}$, in other words, $Z_{s-1}\supseteq Z_{s}$ on $X_s\cong X_{s-1}$. As $D_s=U$ is isotropic, we get $U\subset D_{s-1}^\perp$. Thus, $Z_{s}$ is defined by the condition that $D_s/D_{s-1}\rightarrow D_{s-1}^\perp/(F_{q_s}\cap D_{s-1}^\perp)$ is zero. If $s\leq a$, then $F_{q_s}\subseteq D_{s-1}^\perp$ and if $s>a$, then $D_{s-1}^\perp/(F_{q_s}\cap D_{s-1}^\perp)=D_{\rho_s}^\perp/F_{q_s}$ \cite[\S2.2]{AF}. So, we have the same expression as \eqref{eqn3.60}.
 
 We also note that $D_{\rho_i}^\perp=(V/D_{\rho_i})^\vee$. Combined with the fact that $c_1((D_i/D_{i-1})^\vee)$ acts as $R_i$ on $c_{\widetilde{\lambda}_i}(D_{\rho_i}^\perp-F_{q_i}-D_i/D_{i-1})$ and properties of $T_y(\cdot)$, we express the motivic Hirzebruch class of the bundle $(D_i/D_{i-1})^\vee\otimes U/D_i$ over the motivic Hirzebruch class of $(D_i/D_{i-1})^\vee\otimes D_{\rho_i}^\perp/F_{q_i}$ as
\begin{equation}\label{eqn3.4}
\dfrac{T_y((D_i/D_{i-1})^\vee\otimes U/D_i)}{T_y((D_i/D_{i-1})^\vee\otimes D_{\rho_i}^\perp/F_{q_i})}=\dfrac{T_y(R_i\otimes D_{\rho_i}^\vee)}{T_y(R_i\otimes(V-F_{q_i}-U))\cdot T_y(R_i\otimes D_i)}. 
\end{equation}
As in \cite[pg. 1804]{ACT}, the expression for $T_y(R_i\otimes D_i)$ is given by 
\begin{align*}
T_y(R_i\otimes D_i)&=T_y(R_i-c_1(D_1^\vee))\cdots T_y(R_i-c_1((D_i/D_{i-1})^\vee))\\
&=T_y(R_i-R_1)\cdots T_y(R_i-R_{i-1}),
\end{align*}
and by a similar procedure we arrive at
\begin{align*}
T_y(R_i\otimes D_{\rho_i}^\vee)&=T_y(R_i+c_1(D_1^\vee))\cdots T_y(R_i+c_1((D_{\rho_i}/D_{\rho_i-1})^\vee))\\
&=T_y(R_i+R_1)\cdots T_y(R_i+R_{\rho_i}).
\end{align*}
Additionally, followed by \cite[Lemma 5.16]{HIMN}, \cite[\S 2.2, \S 2.3]{AF}, we obtain
\begin{align*}
\pi_*\left(\prod_{i=1}^s c_{\widetilde{\lambda}_i}(D_{\rho_i}^\perp-F_{q_i}-D_i/D_{i-1})\right)&=\dfrac{\prod_{1\leq j<i\leq s}(1-T_{ij})}{\prod_{1\leq j\leq\rho_i<i\leq s}(1+T_{ij})}\cdot c_{\lambda_1}(1)\cdots c_{\lambda_s}(s)\\
&=\Theta_\lambda^\rho(d(1),\ldots,d(s))
\end{align*}
where $d(i)=c(V-U-F_{q_i})$. Putting these all together establish the statement.
\end{proof}

As a corollary, specializing the sequence $\mathbf{q}$ to consist of all positive integers yields the following formula involving the {\it Pfaffian}. One may see \cite[Appendix A.1]{AF} for the definition of the Pfaffians and their properties. 
\begin{cor}\label{cor3.6}
Let $q_i>0$ for all $i$. Then $\rho_i=i-1$, and thus the class $(\pi\iota)_*T_y(\widetilde{\Omega}_\lambda)$ is given by
\[
\prod_{j\leq i-1}\dfrac{T_y(R_i+R_{j})}{T_y(R_i-R_j)}\prod_{1\leq i\leq s}\dfrac{1}{T_y(R_i\otimes(V-F_{q_i}-U))}\mathrm{Pf}_\lambda(c(1),\ldots,c(s))\scalebox{1.2}{$\cap$}\; T_y(X),
\]
where
$
c(i)=c(V-F_{q_i}-U)
$ for $i=1,\ldots,s$
, if $s$ is even, 
 and if $s$ is odd, 
 \[
c(i)=
\begin{cases}
\dfrac{c(V-F_{q_i}-U)}{T_y(R_i\otimes(V-F_{q_i}-U))}&1\leq i\leq s\\
1& i=s+1
\end{cases}
\]
\end{cor}

\subsubsection{CSM classes of a resolution}
Here we present the CSM classes of the resolution for the isotropic Grassmannian degeneracy loci.  

Let $c(i)_j$ denote the degree $j$ term in $c(i)$, and $t$ be a variable. We define $c_t(i):=\sum_{j\geq0}c(i)_jt^j$. 
Let $\lambda=\lambda(p,\mathbf{q})$ be the partition associated to $p$ and the sequence $\mathbf{q}$, with the isotropic Grassmannian degeneracy loci $\Omega_\lambda^C$. Let $\widetilde{\Omega}_\lambda$ be the resolution of $\Omega_\lambda^C$ and $\iota:\widetilde{\Omega}_\lambda\rightarrow X_s$ be the natural inclusion as before. Then we have the following theorem:
\begin{thm}
Let $s\leq n+1-p$. The class of the resolution $(\pi\iota)_*c_{SM}(\widetilde{\Omega}_\lambda)$ is given by 
\[
\dfrac{\prod_{j\leq \rho_i}1+R_i+R_{j}}{\prod_{j<i}1+R_i-R_j} \prod_{1\leq i\leq s}\dfrac{(1+R_i)^{-ch(i)_0}}{c_{\frac{1}{1+R_i}}(i)}\Theta_\lambda^\rho(c(1),\ldots,c(s))\scalebox{1.2}{$\cap$}\; c_{SM}(X).
\]
Here $c(i)=c(V-F_{q_i}-U)$ for $i=1,\ldots,s$ and $ch(i)_0$ is the virtual rank of $V/F_{q_i}-U$.
\end{thm}
\begin{proof}
We apply Theorem \ref{mainC} with $y=-1$, and use \eqref{ACTlem4.2} to obtain 
\[
c(R_i\otimes (V-F_{q_i}-U))=(1+R_i)^{ch(i)_0}c_{\frac{1}{1+R_i}}(i),
\]
as desired. In particular, $ch(i)_0=p+1_i-2$ if $q_i>0$ and $p+q_i-1$ if $q_i<0$. 
\end{proof}

\subsection{Stratifications}\label{sec3.3}
This section is dedicated to proving Theorem \ref{thm3.10}, which represents the main theorem, Theorem \ref{mainStra} for isotropic degeneracy loci $\Omega_\lambda^C$. 

 For the isotropic degeneracy locus $\Omega_\lambda^C\subseteq X$ associated to 
\[
\mathbf{q}=\left(q_1> q_2>\cdots> q_a>0> q_{a+1}>\cdots> q_s\right)
\quad\text{and}\quad \lambda=\lambda(p,\mathbf{q}),
 \]
we consider a sub-locus 
$\Omega_{{\lambda}^+}\subseteq\Omega_\lambda^C$.
Let $\mathbf{k}$ be a weakly increasing sequence 
\[
\mathbf{k}=\left(0<k_1\leq k_2\leq \cdots\leq k_a\leq k_{a+1}\leq\cdots\leq k_s\right)
\] such that $i\leq k_i\leq \min\{n+1-q_i,n+1-p\}$, $k_{a+1}-k_a\leq -q_{a+1}+\mathrm{max}\{0,q_a+q_{a+1}-1\}$, and
\[
k_i-k_{i-1}+(\rho_{k_{i-1}}-\rho_{k_i})\leq q_{i-1}-q_i
\]
if $i>a+1$.
The partition 
 $\lambda^+=\lambda^+(p,\mathbf{q}, {\mathbf{k}})=(\lambda_1\geq\cdots\geq\lambda_s>0)$ is defined by
\begin{equation}\label{eqn:LamPlus}
\lambda_{k_i}^+=\begin{cases}
q_i+p-1& if\;\; i\leq a\\
q_i+p-1+k_i-\rho_{k_i}& if\;\; i>a,
\end{cases}
\end{equation}
and the remaining parts of $\lambda^+$ are filled minimally so as to be strict if $k<k_a$ and weak if $k>k_a$. If $k_{i}=k_{i+1}$, we set $\lambda_{k_i}=q_i+p-1$ if $i\leq a$, and $\lambda_{k_i}=q_i+p-1+k_i-\rho_{k_i}$ if $i>a$. 
The corresponding locally closed stratum is given by 
\begin{equation}\label{eqn3.5}
\Omega_{\lambda^+}^\circ:=\{x\in X\;|\;\mathrm{dim}(U\cap F_{q_i})|_x= k_i, i=1,\ldots, s\}.
\end{equation}
We write the inequality $\mathbf{k}'>\mathbf{k}$ when $k_i'>k_i$ for all $i$, and let $\lambda'=\lambda^+(p,\mathbf{q},\mathbf{k}')$. It is smooth of pure codimension $|\lambda^+|$ in $X$, satisfying
\begin{equation}\label{eqn3.6}
\Omega_{\lambda^+}^\circ=\Omega_{\lambda^+}\backslash\bigcup_{\mathbf{k}'>\mathbf{k}}\Omega_{\lambda'}
\end{equation}
In addition, the codimension of $\Omega_{\lambda^+}$ in $\Omega_\lambda^C$ is $|\lambda^+|-|\lambda|$, and the codimension one stratum in $\Omega_\lambda^C$ does not exist. 

All subsequent results are established under the assumptions stated above.
Let $X$ be proper. The {\it Hirzebruch $\chi_y$-genus} of $X$ is given by 
\[
\chi_y(X):=\int_XT_y\left(X\right)\scalebox{1.2}{$\cap$}\left[X\right]
\]
in $\mathbb{Q}[y]$. This $\chi_y$-genus recovers classical invariants: it equals the topological Euler characteristic when $y=-1$, the holomorphic Euler characteristic when $y=0$, and the signature when $y=1$, see \cite{BSY,ACT}. We also note that, for singular spaces, these specializations also fit into the general theme that global invariants can often be related to local contributions from singularities. For instance the specialization $y=-1$ is tied to Poincar\'e-Hopf type formulas and local invariants such as the Euler obstruction, while the case $y=1$ is connected with signature theories for singular spaces and singular $L$-classes. 

\begin{thm}\label{thm3.10}
Let $\widetilde{\Omega}_\lambda$ as in \S\ref{sec3.2}, with $\lambda=\lambda(p,\mathbf{q})$. Then we have
\[
(\pi\iota)_*T_y\left(\widetilde{\Omega}_{\lambda}\right)=\sum_{\mathbf{\mathbf{k}}}(-y)^{\overline{\mathbf{k}}}\iota_*T_y\left(\Omega_{\lambda^+}\right),
\]
where $\overline{\mathbf{k}}=\sum_{i=1}^s(k_i-i)$ and $\lambda^+=\lambda^+(p,\mathbf{q},\mathbf{k})$.
\end{thm}

\begin{proof}
By the same argument as in the ordinary case \cite[\S5.1]{ACT}, the map $\phi:\widetilde{\Omega}_\lambda\rightarrow \Omega_\lambda^C$ in \eqref{digm1} is Zariski locally trivial on each locus $\Omega_{\lambda^+}^\circ$. Indeed, for all $x \in \Omega_{\lambda^+}^\circ$, the dimensions of the intersections $\dim(U \cap F_{q_i})_x = k_i$ are constant, and so the intersections $U\cap F_{q_i}$ define nested subbundles of $U$. After restricting to a Zariski open neighborhood, these bundles can be simultaneously trivialized, so the associated relative partial flag variety is locally a product with a fixed partial flag variety. The fibers of $\phi$ on $\Omega_{\lambda^+}^\circ$ are then given by fixed Schubert incidence conditions, and thus $\phi$ is Zariski locally trivial. 

In particular, for a generic point $x\in \Omega_{\lambda^+}$, the fiber of $\phi$ on $\Omega_{\lambda^+}$ is isomorphic to the Schubert variety
\[
\mathbb{S}_{\mu^+}:=\left\{(D_1\;\scalebox{1.2}{$\subseteq$}\; D_2\;\scalebox{1.2}{$\subseteq$}\;\cdots\;\scalebox{1.2}{$\subseteq$}\; D_{s})\;|\; D_j\;\scalebox{1.2}{$\subseteq$}\; K_{k_j}\;\text{for all }j\right\}
\]
which corresponds to a partition $\mu^+$, and is in the ordinary partial flag variety $Fl(1,2,\ldots,s; U)$. Here $K_0\subset K_1\subset \cdots\subseteq K_{n+1-p}$ is a fixed complete flag in $U=U_p$, and the subscript of $K_i$ indicates its dimension, i.e., $\mathrm{dim}\;K_i=i$. Notably, the subspaces $K_{k_i}$ can be identified with the fibers $(U\cap F_{q_i})|_x$ at a generic point $x\in \Omega_{\lambda^+}$. 

Since $\mathbb{S}_{\mu^+}$ is a type A Schubert variety, the remaining steps of the argument follow directly from the proof in \cite[Theorem 5.1]{ACT}, with some minor notational modifications. Specifically, 
the Schubert cell $\mathbb{S}_{\mu^+}^\circ\subset \mathbb{S}_{\mu^+}$ is shown to be isomorphic to the affine space $\mathbb{A}^{\sum_{i}^s(k_i-i)}$ through explicit coordinate computations. In particular, when $\mathbf{k}$ is strictly increasing, $\mathbb{S}_{\mu^+}$ corresponds to the partition $\mu^+=(k_s-s,\ldots, k_1-1)$.

By applying Lemma \ref{lem:fib}, we obtain
\[
(\pi\iota)_*T_y\left(\widetilde{\Omega}_{\lambda}\right)=\sum_{\mathbf{\mathbf{k}}}d_{(\mathbf{k})}\iota_*T_y\left(\Omega_{\lambda^+}\right),
\]
where $d_{(\mathbf{k})}$ represent the Hirzebruch $\chi_y$-genus of the Schubert cell $\mathbb{S}_{\mu^+}^\circ$. The definition of the Hirzebruch $\chi_y$-genus can be found in \cite{BSY}. Since $\mathbb{S}_{\mu^+}^\circ$ is isomorphic to $\mathbb{A}^{\sum_{i}^s(k_i-i)}$, we conclude that $d_{(\mathbf{k})}=(-y)^{\sum_{i}^s(k_i-i)}$, and the desired result follows.
\end{proof}

We now present an example that demonstrates Theorem \ref{thm3.10}.
\begin{ex}\label{ex3.11}
For $n=6$, $p=3$ and $\mathbf{q}=(5,2,\bar{1},\bar{4})$, with the partition $\lambda(p,\mathbf{q})=(7,4,2,1)$, the class $(\pi\iota)_*T_y(\widetilde{\Omega}_\lambda)$ is given, in terms of the stratification, by the following expression:
\begin{align*}
(\pi\iota)_*T_y(\widetilde{\Omega}_\lambda)=&\iota_*T_y(\Omega_{\mathbf{k}_0})-y\iota_*T_y(\Omega_{\mathbf{k}_1})-y\iota_*T_y(\Omega_{\mathbf{k}_4})+y^2\iota_*T_y(\Omega_{\mathbf{k}_2})\\
&+y^2\iota_*T_y(\Omega_{\mathbf{k}_5})-y^3\left(\iota_*T_y(\Omega_{\mathbf{k}_3})+\iota_*T_y(\Omega_{\mathbf{k}_6})\right)+y^4\iota_*T_y(\Omega_{\mathbf{k}_7}).
\end{align*}
Here we use $\mathbf{k}_i$ instead of $\lambda^+(p,\mathbf{q},\mathbf{k})$ for simplicity, and these $\mathbf{k}_i$ are the following:
\begin{align*}
\mathbf{k}_0=(1,2,3,4)\quad\mathbf{k}_1=(1,3,3,4)\quad\mathbf{k}_2=(1,3,4,4)\quad\mathbf{k}_3=(1,4,4,4)\\
\mathbf{k}_4=(2,2,3,4)\quad\mathbf{k}_5=(2,3,3,4)\quad\mathbf{k}_6=(2,3,4,4)\quad\mathbf{k}_7=(2,4,4,4)
\end{align*}
\end{ex}
The following is an example illustrating Theorem \ref{thm3.10} and from it one can calculate the motivic class of the locus $\Omega_{\lambda}$ associated to the partition $\lambda=(4,2,1)$ as below.  

\begin{ex}\label{main_illu}
We consider $\widetilde{\Omega}_{(4,2,1)}$ with $n=4,$ $p=2$ and $\mathbf{q}=(3,1,\overline{2})$. The strata for this locus $\Omega_{(4,2,1)}$ are $\Omega_{\lambda_0}$, $\Omega_{\lambda_1}$, $\Omega_{\lambda_2}$ and $\Omega_{\lambda_3}$, where 
\begin{align*}
\lambda_0=(4,2,1)\;\text{with}\; \mathbf{k}_1=(1,2,3),&\quad \lambda_1=(4,3,2)\;\text{with}\;\mathbf{k}_1=(1,3,3),\\
 \lambda_2=(5,4,1)\;\text{with}\;\mathbf{k}_2=(2,2,3),&\quad\lambda_1=(5,4,2)\;\text{with}\;\mathbf{k}_3=(2,3,3).
 \end{align*} Since the locus $\Omega_{\lambda_1}$ is associated to $\mathbf{q}'=(1,2,3)$ and $\mathbf{k}'=(3,2,1)$, its strata are given by $\Omega_{\lambda_1},\Omega_{\lambda_2}$ and $\Omega_{\lambda_4}$ where $\lambda_4=(4,3)$ which is from $\mathbf{k}_4=(2,3,3)$, and $\mathbf{q}'$. Applying Theorem \ref{thm3.10}, we have the following equations:
  \begin{align*}
  (\pi\iota)_*T_y(\widetilde{\Omega}_{\lambda_0})&=\iota_*T_y(\Omega_{\lambda_0})-y\iota_*T_y(\Omega_{\lambda_1})-y\iota_*T_y(\Omega_{\lambda_2})+y^2\iota_*T_y(\widetilde{\Omega}_{\lambda_3}),\\
    (\pi\iota)_*T_y(\widetilde{\Omega}_{\lambda_1})&=\iota_*T_y(\Omega_{\lambda_1})-y\iota_*T_y(\Omega_{\lambda_2})+y^2\iota_*T_y(\Omega_{\lambda_4}),\\
    (\pi\iota)_*T_y(\widetilde{\Omega}_{\lambda_2})&=\iota_*T_y(\Omega_{\lambda_2}), \;(\pi\iota)_*T_y(\widetilde{\Omega}_{\lambda_3})=\iota_*T_y(\Omega_{\lambda_3}),\;\text{and}\\
    (\pi\iota)_*T_y(\widetilde{\Omega}_{\lambda_4})&=\iota_*T_y(\Omega_{\lambda_4}).\\
  \end{align*}
  
    Then the motivic Hirzebruch class of $\Omega_{\lambda}$ is given by
    \begin{align}
  \begin{aligned}\label{eq:stra}
\iota_*T_y(\Omega_{\lambda})= &(\pi\iota)_*T_y(\widetilde{\Omega}_{\lambda_0})+y (\pi\iota)_*T_y(\widetilde{\Omega}_{\lambda_1})+(y^2+y) (\pi\iota)_*T_y(\widetilde{\Omega}_{\lambda_2})\\
&-y^2 (\pi\iota)_*T_y(\widetilde{\Omega}_{\lambda_3})-y^3 (\pi\iota)_*T_y(\widetilde{\Omega}_{\lambda_4}),
\end{aligned}
\end{align}
as described in Theorem \ref{mainCal}.
\end{ex}
Theorem \ref{mainC} enables us to express \eqref{eq:stra} on the right as a polynomial in the variable $y$ with coefficient given by the Chern classes of the given vector bundles. 
In addition, Theorem \ref{thm3.10} can be applied to compute the CSM classes of Schubert varieties in Lagrangian Grassmannian as follows.

\begin{ex}\label{ex3.12}
Let us consider the Lagrangian Grassmannian $X=LG(4,8)$ and a Schubert variety associated to a partition $\lambda=(4,2)$ with $p=1$ and $\mathbf{q}=(4,2)$. By the stratification, it has two strata: the stratum with $\lambda=(4,2)$  
and $\mathbf{k}=(1,3)$ (or $\lambda^+=(4,3,2)$). The formula for the degeneracy loci (or Schubert variety) associated to $\lambda=(4,2)$ results in the following computation 
\begin{align*}
\iota_*c_{SM}(S_\lambda)&=\left(
{\Large\Tableau[p]{&&&\\ ~&&}-2\;\Tableau[p]{&&&\\ ~&&\\~&~&}-5\;
\Tableau[p]{&&&\\~&&&}+11\;
\Tableau[p]{&&&\\~&&&\\~&~&&~}-19\;
\Tableau[p]{&&&\\~&&&\\~&~&&}+12\;\Tableau[p]{&&&\\~&&&\\~&~&&\\~&~&~&}}\;
\right)\scalebox{1.2}{$\cap$}\; c_{SM}(X)\\
&={\Large\Tableau[p]{&&&\\ ~&&}+
3\;\Tableau[p]{&&&\\ ~&&\\~&~&}+5\;
\Tableau[p]{&&&\\~&&&}+14\;
\Tableau[p]{&&&\\~&&&\\~&~&&~}+19\;
\Tableau[p]{&&&\\~&&&\\~&~&&}+6\;\Tableau[p]{&&&\\~&&&\\~&~&&\\~&~&~&}}\;.
\end{align*}
The last equality is from the Littlewood--Richardson rule for the Schubert structure coefficients of isotropic Grassmannians and the result coincides with the one computed by \cite{AM16}.
\end{ex}

We remark that the coefficients of the Schubert expansion of the CSM classes for Schubert varieties are non-negative: the case of the type A Grassmannian Schubert cells were proved in \cite{Huh16}, and that of Schubert cells in any homogeneous space $G/P$ are proved in \cite{AMSS23}.

\section{Odd Orthogonal Degeneracy Loci}\label{sec4}

Let $V$ be a rank $2n+1$ vector bundle on $X$, equipped with a non-degeneracy symmetric form. We take a flag of isotropic subbundles
\[
F_\bullet: 0\;\scalebox{1.2}{$\subset$}\; F_n\;\scalebox{1.2}{$\subset$}\; F_{n-1}\;\scalebox{1.2}{$\subset$}\; \cdots\;\scalebox{1.2}{$\subset$}\; F_1\;\scalebox{1.2}{$\subset$}\; F_0\;\scalebox{1.2}{$\subset$}\; F_{\bar{1}}\;\scalebox{1.2}{$\subset$}\;\cdots \;\scalebox{1.2}{$\subset$}\; F_{\bar{n}}= V
\]
on $X$ where $F_i$ are all isotropic with respect to the symmetric form, and rk$(F_i)=n+1-i$, $F_{\bar{i}}=(F_{i+1})^\perp$ for all $i$. We note that $F_1$ is a maximal isotropic subbundle of $V$. Let $U_p$ be an isotropic subbundle of $V$ on $X$ where the rank of $U_p$ is $n+1-p$ for some $1\leq p\leq n$.

A partition $\lambda=\lambda(p,\mathbf{q})$ and the sequence $\rho$ for this odd orthogonal case are defined in the same way as in the isotropic case.
The {\it odd orthogonal degeneracy locus} is 
\[
\Omega_\lambda^B:=\left\{ x\in X\;|\;\mathrm{dim}(U\cap F_{q_i})|_x\geq i, i=1,\ldots, s\right\}.
\]

The proof unfolds in the same way as in the isotropic degeneracy loci case. The fundamental computation of the orthogonal Grassmannian bundles and the relevant basic case are addressed first, and the motivic Hirzebruch class of the push-forward of the resolution and stratifications are then modified for this setting, thereby completing Theorem \ref{mainStra}. In the last part of this section, we discuss the motivic Hirzebruch class of the even maximal orthogonal degeneracy loci of type D.

\subsection{Motivic classes}
We say that a vector bundle $G$ of rank $t$ is isotropic if $t<n$ and coisotropic if $t>n$.  

Let $L$ be a line bundle, $E\subset V$ a subbundle of rank $e$, and $F\subset V$ a maximal isotropic bundle of rank $n$ on $X$. Let $M$ be the line bundle defined by
\[
M\cong F^\perp/F.
\] 
We will use the fact that for any isotropic subbudle $D\subset V$, $D^\perp$ can be identified with $(V/D)^\vee$ by the symmetric form.

The lemma below serves as the base case, corresponding to the degeneracy loci defined by the condition $L\subseteq E$:
\begin{lemma}\label{lem4.2}
Let $\iota: Z\hookrightarrow X$ be the locus where $L\subseteq E$.
If $e<n$, then its motivic Hirzebruch class is
\begin{equation}\label{eqn4.1}
2\iota_*T_y(Z)=\dfrac{1}{T_y((V/F^\perp\oplus F/E)\otimes R)}c_{2n-e}(V-E-L-M)\;\scalebox{1.2}{$\cap$}\; T_y(X).
\end{equation}
If $e>n$, we have
\[
\iota_*T_y(Z)=\dfrac{1}{T_y(V/E\otimes R)}c_{2n+1-e}(V-E-L)\;\scalebox{1.2}{$\cap$}\; T_y(X).
\]
Here $R=c_1(L^\vee)$ is the raising operator.
\end{lemma}
\begin{proof}
The case of $e<n$ follows by \cite[Lemma 2.2]{HIMN}, and Lemma \ref{lem2.1}. 
When $e>n$, $E$ is co-isotropic, so that the locus is (scheme-theoretically) defined by the vanishing of $L\rightarrow V/E$ (See \cite[\S 3.1]{AF}). As in the isotropic case, we have the second result, combined with Lemma \ref{lem2.1} and the zeros of the section $L^\vee\otimes V/E$.
\end{proof}

 In fact, we have $M\cong\mathrm{det}(V)$ and an isomorphism between the trivial line bundle and $M^2\cong M\otimes M$ induced by the non-degenerate symmetric form of $V$. This implies $c_1(M)=0$ in $A_*(X)[\frac{1}{2}]$.

\subsection{Resolution of singularities}
The construction of the resolution of singularities for the orthogonal degeneracy loci follows the same description as that for the isotropic degeneracy loci, except that we now consider isotropic subbundles of a rank $2n+1$ vector bundle over a variety $X$ equipped with a non-degenerate symmetric form (\cite[\S3, p.1749]{AF},\cite[\S6.3]{HIMN}, \cite[\S3]{A}). 
That is, the varieties $X_i = \mathbb{P}(U/D_{i-1})$ form an iterated sequence of projective bundles with natural projections $\pi_i: X_i \to X_{i-1}$, and we denote the composition by $\pi = \pi_s \circ \cdots \circ \pi_1$. 
The locus $Z_i$ is defined by the conditions $D_j \subseteq F_{q_j}$ for $1 \leq j \leq i$, with $i = 1, \ldots, s$, as before. We denote $Z_s = \widetilde{\Omega}_\lambda$ as the resolution of $\Omega_\lambda^B$, and set $Z_0 = X$.

Let $\iota:\widetilde{\Omega}_\lambda\rightarrow X_s$ be the natural inclusion. Suppose that $a$ is the maximal number where $q_i>0$ for all $i\leq a$. 
\begin{thm}\label{mainB}
The class of the push-forward of the resolution $(\pi\iota)_*T_y(\widetilde{\Omega}_\lambda)$ is given by 
\[
\dfrac{1}{2^a}\dfrac{\prod_{j\leq \rho_i}T_y(R_i+R_{j})\prod_{i=1}^aT_y(R_i)}{\prod_{j<i}T_y(R_i-R_j)}\Theta_\lambda^\rho(d(1),\ldots,d(s))\scalebox{1.2}{$\cap$}\; T_y(X),
\]
where $d(i)=c(V-F_{q_i}-U-M)/T_y(R_i\otimes(V-F_{q_i}-U-M))$ for $i=1,\ldots,a$ and $d(i)=c(V-F_{q_i}-U)/T_y(R_i\otimes(V-F_{q_i}-U))$ for $i=a+1,\ldots,s$. 
\end{thm}

\begin{proof}
Lemma \ref{lem4.2} together with the discussion above Lemma \ref{lem3.5} results in 
\begin{align*}
2^a\iota_*T_y(\widetilde{\Omega})=&\left(\displaystyle\prod_{i=1}^a\dfrac{T_y((D_i/D_{i-1})^\vee\otimes U/D_i)}{T_y((D_i/D_{i-1})^\vee\otimes (D_{\rho_i}^\perp/F^\perp\oplus F/F_{q_i}))}c_{\widetilde{\lambda}_i}(D_{\rho_i}^\perp-F_{q_i}-D_i/D_{i-1}-M)\right.\\
&\left.\cdot\displaystyle\prod_{i=a+1}^s\dfrac{T_y((D_i/D_{i-1})^\vee\otimes U/D_i)}{T_y((D_i/D_{i-1})^\vee\otimes (D_{\rho_i}^\perp/F_{q_i}))}c_{\widetilde{\lambda}_i}(D_{\rho_i}^\perp-F_{q_i}-D_i/D_{i-1})
\right)
\scalebox{1.2}{$\cap$}\;\pi^*T_y(X)
\end{align*}
where $\rho_i$ and $\widetilde{\lambda}_i$ are defined as in the isotropic case. 
Then for $1\leq i\leq a$, we have
\begin{equation}\label{eqn4.3a}
\dfrac{T_y((D_i/D_{i-1})^\vee\otimes U/D_i)}{T_y((D_i/D_{i-1})^\vee\otimes (D_{\rho_i}^\perp/F^\perp\oplus F/F_{q_i}))}=\dfrac{T_y(R_i\otimes D_{\rho_i}^\vee)T_y(R_i\otimes M)}{T_y(R_i\otimes (V-U-F_{q_i}))T_y(R_i\otimes D_i)}
\end{equation}
replacing $(D_i/D_{i-1})^\vee$ with $R_i$. Since $c_1(M)=0$, $T_y(R_i\otimes M)$ becomes $T_y(R_i)$. 
Similarly, for $a+1\leq i\leq s$, we obtain
\begin{equation}\label{eqn4.3b}
\dfrac{T_y((D_i/D_{i-1})^\vee\otimes U/D_i)}{T_y((D_i/D_{i-1})^\vee\otimes (D_{\rho_i}^\perp/F_{q_i}))}=\dfrac{T_y(R_i\otimes D_{\rho_i}^\vee)}{T_y(R_i\otimes (V-U-F_{q_i}))T_y(R_i\otimes D_i)}.
\end{equation}
The proof thereafter proceeds by parallel arguments with \cite[\S3.3]{AF} as in the isotropic case.
\end{proof}

The following corollary is the special case of Theorem \ref{mainB} where all $q_i$ are positive.
\begin{cor}\label{cor4.5}
Let $q_i>0$ for all $i=1,\ldots,s$. Then $\rho_i=i-1$ so that
\[
2^s(\pi\iota)_*T_y(\widetilde{\Omega}_\lambda)=\dfrac{\prod_{j\leq \rho_i}T_y(R_i+R_{j})\prod_{i=1}^sT_y(R_i)}{\prod_{j<i}T_y(R_i-R_j)} \mathrm{Pf}_\lambda(d(1),\ldots,d(s))\scalebox{1.2}{$\cap$}\; T_y(X),
\]
where
$
d(i)=c(V-F_{q_i}-U-M)/T_y(R_i\otimes(V-F_{q_i}-U-M))
$ for $i=1,\ldots,s$. 
\end{cor}

Another specialization of Theorem \ref{mainB} is when $y=-1$. In this event,  we recover the CSM classes of the resolution $\widetilde{\Omega}_\lambda$ as follows:

\begin{thm}
Let $\widetilde{\Omega}_\lambda$ be the resolution of $\Omega_\lambda^B$ and $\iota:\widetilde{\Omega}_\lambda\rightarrow X_s$ be the natural inclusion as before. Suppose that $a$ is the maximal number where $q_i>0$ for all $i\leq a$. 
Then the class of the resolution $(\pi\iota)_*T_y(\widetilde{\Omega}_\lambda)$ is given by 
\[
\dfrac{1}{2^a}\dfrac{\prod_{j\leq \rho_i}(1+R_i+R_{j})\prod_{i=1}^a(1+R_i)}{\prod_{j<i}(1+R_i-R_j)} \prod_{1\leq i\leq s}\dfrac{(1+R_i)^{-ch(i)_0}}{c_{\frac{1}{1+R_i}(i)}}\Theta_\lambda^\rho(c(1),\ldots,c(s))\scalebox{1.2}{$\cap$}\; T_y(X),
\]
where $c(i)=c(V-F_{q_i}-U-M)$ for $i=1,\ldots,a$, $c(i)=c(V-F_{q_i}-U)$ for $i=a+1,\ldots, s$ and $ch(i)_0$ are the virtual ranks of $V/F_{q_i}-U$ for $i=1,\ldots, a$ and $V/F_{q_i}-U$ for $i=a+1,\ldots,s$.
\end{thm}

\subsection{Stratifications}
Let $\Omega_\lambda^B\subseteq X$ denote the odd orthogonal Grassmannian degeneracy locus associated with the partition
$
\lambda=\lambda(p,\mathbf{q}).
$
The partition $\lambda^+(p,\mathbf{q},\mathbf{k})$ coincides with the one defined in the isotropic case. Also, a sub-locus $\Omega_{\lambda}^+$ can be considered and $\Omega_{\lambda^+}^\circ$ admits the same description as given in \eqref{eqn3.5} and \eqref{eqn3.6}. 
Since the fiber of the restriction map $\phi:\widetilde{\Omega}_\lambda\rightarrow \Omega_\lambda^B$ and the description of $\Omega_\lambda^\circ$ are the same as in the isotropic case, the class $(\pi\iota)_*T_y(\widetilde{\Omega}_\lambda)$ can be expressed as in Theorem \ref{thm3.10}. 

The main differences from the isotropic case is the class $T_y(\Omega_{\mathbf{k}})$ of each stratum $\Omega_{\mathbf{k}}$. Thus we may look into the following example on the odd orthogonal Grassmannians.

\begin{ex}\label{ex4.7}
Let $X=OG(4,9)$ denote the odd maximal orthogonal Grassmannian, and we consider a Schubert variety associated to a partition $\lambda=(4,2)$ with $p=1$ and $\mathbf{q}=(4,2)$. By the stratification, it has two strata: the stratum with $\lambda=(4,2)$ and $\mathbf{k}=(1,3)$ (or $\lambda^+=(4,3,2)$). The formula for the odd orthogonal degeneracy loci (or Schubert variety) associated to $\lambda=(4,2)$ results in the following computation 
\begin{align}
\iota_*c_{SM}(S_\lambda)&=\left(
{\Large
\Tableau[p]{&&&\\ ~&&}-4\;\Tableau[p]{&&&\\ ~&&\\~&~&}-4\;
\Tableau[p]{&&&\\~&&&}+18\;
\Tableau[p]{&&&\\~&&&\\~&~&&~}-28\;
\Tableau[p]{&&&\\~&&&\\~&~&&}+28\;\Tableau[p]{&&&\\~&&&\\~&~&&\\~&~&~&}}\;
\right)\scalebox{1.2}{$\cap$}\; c_{SM}(X) \nonumber \\
&={\Large\Tableau[p]{&&&\\ ~&&}+
4\;\Tableau[p]{&&&\\ ~&&\\~&~&}+4\;
\Tableau[p]{&&&\\~&&&}+14\;
\Tableau[p]{&&&\\~&&&\\~&~&&~}+14\;
\Tableau[p]{&&&\\~&&&\\~&~&&}+8\;\Tableau[p]{&&&\\~&&&\\~&~&&\\~&~&~&}}\;.\label{eqn:4.5a}
\end{align}
where the last equality uses the fact that 
\[
c_{SM}(OG(4,9))=\emptyset+{\Large
8\;\Tableau[p]{&~}+30\;\Tableau[p]{&}+70 \;\Tableau[p]{& &}+68\;\Tableau[p]{&\\~&}+110\;\Tableau[p]{& &&}+216\;\Tableau[p]{& &\\ ~&}+\cdots
}
\]
 and the Littlewood--Richardson rule for the Schubert structure coefficients of odd maximal orthogonal Grassmannians. We note that \eqref{eqn:4.5a} can be recovered via \cite{AM16}, following the approach outlined therein.
 
\end{ex}

\subsection{Even Orthogonal Degeneracy Loci}
In this section we investigate the connection between the even and odd maximal orthogonal degeneracy loci. This discussion is inspired by \cite[\S 6]{ACT}, which treats the case of the ordinary Grassmannian.

Let $V_B$ be a rank $2n+1$ vector bundle over $X$, together with a non-degenerate symmetric form $\langle\cdot ,\cdot\rangle_B$. We consider a partial flag 
\[
F_{\bf{q}}: 0\;\scalebox{1.2}{$\subset$}\; F_{q_1}\;\scalebox{1.2}{$\subset$}\; F_{q_2}\;\scalebox{1.2}{$\subset$}\; \cdots\;\scalebox{1.2}{$\subset$}\; F_{q_s}\;\scalebox{1.2}{$\subset$}\; V_B
\]
of isotropic subbundles and the maximal isotropic subbundle $U$ on $X$, where $\mathrm{rk}(U)=n$, $\mathrm{rk}(F_{q_i})=n+1-q_i$, and all $q_i$ are nonnegative.

The odd (maximal) orthogonal degeneracy locus is given by
\begin{equation}\label{eqn4.4}
\Omega_\lambda^B:=\left\{ x\in X\;|\;\mathrm{dim}(U\cap F_{q_i})|_x\geq i, i=1,\ldots, s\right\},
\end{equation}
with the partition $\lambda$ defined by $\lambda_i=q_i$ as in the isotropic case. 
Let $\pi_B:OG(V_B):=OG^1(V_B)\rightarrow X$ be the odd maximal orthogonal Grassmannian bundle. Let $\mathfrak{S}\subset V_B$ be the tautological subbundle on $OG(V_B)$. ($V_B$ should be understood as $\pi_B^*(V_B)$. We omit notation for such pullbacks.)
We define the sub-locus $\widehat{\Omega}_{\lambda}^B\subset OG(V_B)$ by the conditions 
\[
\mathrm{dim}(\mathfrak{S}\cap F_{q_i})\geq i\quad\text{for $1\leq i\leq s$}.
\]
 Then we have the map $\phi_B:\widehat{\Omega}_\lambda^B\rightarrow \Omega_\lambda^B$ such that the fiber of $\phi_B$ over a point $x\in X$ is given by
\[
\left\{{S}\in OG(V_B|_x)\;|\;\mathrm{dim}({S}\cap F_{q_i})|_x\geq i\;\text{for}\;i=1,\ldots,s\right\},
\]
and it produces the following diagram
\[
\begin{tikzcd}
\widehat{\Omega}_\lambda^B \arrow[d,"\phi_B"] \arrow[r,hook,"\iota_B"] &OG(V_B)  \arrow[d,"\pi_B"]\\
 \Omega_\lambda^B  \arrow[r, hook]  &X.
\end{tikzcd}
\]
\subsubsection{The stratification and motivic Hirzebruch class of $\widehat{\Omega}_\lambda^B$}
The map $\phi_B:\widehat{\Omega}_\lambda^B\rightarrow \Omega_\lambda^B$ gives rise to a stratification of $\Omega_\lambda^B$, which we now describe.

\begin{defi}
Let $\mathbf{g}=(0<g_1\leq g_2\leq\cdots\leq g_s)$ be a weakly increasing sequence such that $i\leq g_i\leq n.$ The function $\nu$ is defined on the set of sequences such that $\nu(\mathbf{g})$ is the partition defined as the dual of the strict partition $\widetilde{\nu}=\widetilde{\nu}(\mathbf{g})$ in the shape $(n,n-1,\ldots,1)$, where $\widetilde{\nu}$ is given by 
\[
\widetilde{\nu}_{i}=n+1-g_i
\]
 for each $i$, and if $g_i=g_{i+1}$, then take $\widetilde{\nu}_{i+1}=n+1-g_{i+1}$, with the remaining parts filled in to form a strict partition. (Similarly, if it ends up with $\widetilde{\nu}_i=\widetilde{\nu}_{i+1}$, then we keep $\widetilde{\nu}_{i+1}$ and set $\widetilde{\nu}_{i}=\widetilde{\nu}_{i+1}+1$ to form a strict partition.)
 \end{defi}

Let us consider the locally closed stratum $\Omega_{\lambda^+}^B(\subset \Omega_\lambda^B)$ given by the same equation \eqref{eqn3.5} and property \eqref{eqn3.6} associated to a weakly increasing sequence $\mathbf{k}=(0<k_1\leq k_2 \leq \cdots\leq k_s)$ and $\lambda^+=\lambda^+(p,\mathbf{q},\mathbf{k})$. Over a general point $x\in X$, the fiber of $\phi_B$ in $\widehat{\Omega}_\lambda^B$ can be seen as the Schubert variety 
\begin{equation}\label{eqn:fiber}
\mathbb{S}_{\nu(\mathbf{k})}=\left\{{S}\in OG(V_B|_x)\;\bigm|\;\mathrm{dim}({S}\cap K_{k_i})|_x\geq i\;\text{for all}\; i\right\},
\end{equation}
associated to a partition $\nu(\mathbf{k})$ where each $K_{k_i}=U\cap F_{q_i}$ has rank $k_i$.

We say that a sequence $\mathbf{a}$ is contained in $\mathbf{b}$, denoted by $\mathbf{a}=(a_1,\ldots,a_s)\leq\mathbf{b}=(b_1,\ldots,b_s)$, if and only if $a_i\leq b_i$, and  $\mathbf{a}<\mathbf{b}$ if and only if $a_i\leq b_i$ with at least one strict inequality. 

\begin{defi}\label{def:beta}
Given $\mathbf{k}$, we define $\boldsymbol{\beta}:=\boldsymbol{\beta}(\mathbf{k})=(\beta_1\leq \cdots\leq\beta_s)$ to be minimal in the component-wise order, with $i\leq\beta_i\leq k_i$, such that $\beta_s=k_s$ and
\[
\beta_i=k_i\quad\text{ if}\quad q_{i}+k_i\geq q_{{i+1}}+k_{i+1}+1.
\]
\end{defi}
The following lemma describes the minimal sequence that does not lie below any sequence producing a strict partition in $\lambda^+$:
\begin{lemma}\label{l:4.8}
The sequence $\boldsymbol{\beta}$ is the smallest among those sequences that are not less than or equal to any sequence $\boldsymbol{\epsilon}=(\epsilon_1\leq\cdots\leq \epsilon_s)$ with $i\leq \epsilon_i\leq k_i$, for which $\xi:=\lambda^+(p,\mathbf{q},\boldsymbol{\epsilon})$ is a strict partition inside $\lambda^+$.
\end{lemma}
\begin{proof}
We first claim that $\mathbf{k}$ is the only sequence containing $\boldsymbol{\beta}$ that gives rise to the strict partition $\lambda^+$.

Suppose that there is a sequence $\boldsymbol{\beta}'=(\beta_1',\ldots,\beta_s')$ containing $\boldsymbol{\beta}$ such that $\beta_i'\geq \beta_i$ for all $i$, and $\lambda^+(p,\mathbf{q},\boldsymbol{\beta}')$ is strict. Since $\beta_s=k_s$ and $\beta_i'\leq k_i$ for all $i$, it follows that $\beta_s'=k_s$. Similarly, if $\beta_i=k_i$, then $\beta_i'=k_i$. So, we focus on the components where $\beta_i<k_i$. In other words, if we take $i<s$ such that $\beta_i<k_i$ and $\beta_j=k_j$ for $j>i$, then by definition of $\boldsymbol{\beta}$, we have 
\begin{equation}\label{eqn4.5}
q_i+k_i<q_{i+1}+k_{i+1}+1.
\end{equation}
 Since the partition $\lambda^+(p,\mathbf{q},\boldsymbol{\beta}')$ is strict, the following equality 
$
q_i-q_{i+1}\geq \beta_{i+1}'-\beta_i',
$ and thus 
\[
q_i+\beta'_i\geq q_{i+1}+\beta_{i+1}'
\]
is satisfied. Since $q_{i+1}+\beta_{i+1}'=q_{i+1}+k_{i+1}$, and by \eqref{eqn4.5}, one has
\[
q_i+k_i\geq q_{i}+\beta_i'\geq q_{i+1}+k_{i+1}\geq q_i+k_i
\]
This implies $\beta_i'=k_i$, and hence, we obtained the first claim. 

Now, we want to show that $\boldsymbol{\beta}$ is the smallest in the sense that any $\boldsymbol{\beta}'<\boldsymbol{\beta}$ is contained in a sequence $\epsilon$ whose associated partition $\xi$ is strict.

We fix the index $i\leq s$ such that $\beta_i=k_i$ and $q_i+k_i\geq q_{i+1}+k_{i+1}+1$. Then we consider the sequence $\boldsymbol{\beta}'$ given by $\beta'_i=\beta_i-1$ and $\beta_j'=\beta_j$ for $j\neq i$. 
Then $\boldsymbol{\beta}'$ is contained in 
\[
\boldsymbol{\epsilon}=(\epsilon_i,\ldots, \epsilon_{i-1}, k_i-1,k_{i+1},\ldots,k_s),
\]
where $\epsilon_j=\min\{k_j,k_i-1\}$ for $1\leq j\leq i-1$. In addition, the associated partition $\xi=\lambda^+(p,\mathbf{q},\boldsymbol{\epsilon})$ is strict, because of the condition $q_i+k_i\geq q_{i+1}+k_{i+1}+1$, and the fact that $\lambda^+$ associated to $\mathbf{k}$ is strict.
\end{proof}

Given this stratification of $\Omega_\lambda^B$, the motivic Hirzebruch class of $\widehat{\Omega}_\lambda^B$ is given by the following:
\begin{prop}
With the above setup, we have 
\[
(\pi_B\iota_B)_*T_y(\widehat{\Omega}_\lambda^B)=\sum_{\mathbf{k}}d_{\mathbf{k}}\iota_*T_y(\Omega_{\lambda^+}^B)\]
where the sums over $\mathbf{k}$ such that its corresponding partition $\lambda^+(p,\mathbf{q},\mathbf{k})$ is strict and $d_{\mathbf{k}}$ denote the sum
\[
d_{\mathbf{k}}:=\displaystyle\sum_{\substack{{\nu}(\boldsymbol{\beta})\subseteq{\nu}'\subseteq {\nu}(\mathbf{k})\\
                  }}
        (-y)^{|{\nu}'|}.
\]
over strict partitions ${\nu}'$ inside ${\nu}(\mathbf{k})$ and containing ${\nu}(\boldsymbol{\epsilon})$.
\end{prop}
\begin{proof}
For each strata $\Omega^B_{\lambda^+}$, one can write
$
(\pi_B\iota_B)_*T_y(\widehat{\Omega}_\lambda^B)=\sum_{\mathbf{k}}d_{\mathbf{k}}\iota_*T_y(\Omega_{\lambda^+}^B)$
with some coefficients $d_{\mathbf{k}}$. 

The fiber of $\phi_B$ over a general point in $\Omega_{\lambda^+}^B$ is considered as the Schubert variety $\mathbb{S}_{\nu(\mathbf{k})}$ in $OG(\mathbb{C}^{2n+1})$, and thus, by Lemma \ref{lem:fib}, we have
\begin{equation}\label{e:4.6}
d_{\mathbf{k}}=\chi_y(\mathbb{S}_{\nu(\mathbf{k})})-\sum_{\boldsymbol{\epsilon}<\mathbf{k}}d_{\boldsymbol{\epsilon}},
\end{equation}
where the sums over $\boldsymbol{\epsilon}< \mathbf{k}$ such that $\xi:=\lambda^+(p,\mathbf{q},\boldsymbol{\epsilon})$ is strict and $\Omega_{\lambda^+}^B\subset\Omega_{\xi}$.

The Schubert variety $\mathbb{S}_{\nu(\mathbf{k})}$ is the union of Schubert cells corresponding to partitions $\nu'\subseteq \nu(\mathbf{k})$.
And, the sum of $d_\epsilon$ over $\boldsymbol{\epsilon}<\mathbf{k}$ is equal to the sum of the Hirzebruch $\chi_y$-genera of the Schubert cells corresponding to strict partitions that are contained in the strict partition $\nu(\boldsymbol{\epsilon})$ for some $\boldsymbol{\epsilon}<\mathbf{k}$ such that $\xi=\lambda^+(p,\mathbf{q},\boldsymbol{\epsilon})$ is strict. 

By \eqref{e:4.6} and Lemma \ref{l:4.8}, $d_\mathbf{k}$ can be expressed as the sum of the Hirzebruch $\chi_y$-genera of Schubert cells corresponding to strict partitions $\nu'$ containing $\nu(\boldsymbol{\beta})$ and contained in $\nu(\mathbf{k})$. Additionally, a Schubert cell corresponding to a strict partition $\nu'$ is isomorphic to the affine space 
 $\mathbb{A}^{|\nu'|}_{\mathbb{C}}$. Hence its Hirzebruch $\chi_y$-genus is $(-y)^{|{\nu}'|}$, which implies the statement.
\end{proof}

\begin{ex}
Specializing to $y=-1$, one compute explicitly the topological Euler characteristic of the fibers \eqref{eqn:fiber} of $\phi_B$, as follows.

Each fiber can be identified with the Schubert variety $\mathbb{S}_{\nu(\mathbf{k})}$ in the orthogonal Grassmannian $OG(n,2n+1)$ associated to the strict partition ${\nu}={\nu}(\mathbf{k})$. So, its topological Euler characteristic can be computed by 
 \[
 \chi(\mathbb{S}_{\nu(\mathbf{k})})=\left|\binom{{\nu}_{j}+j-i+1}{1+j-i}\right|_{1\leq i,j\leq s}+\sum_{1\leq k\leq s-1}\left|\binom{k-i+2}{1+j-i}\right|_{1\leq i,j\leq k},
 \]
 as the number of strict partitions inside ${\nu}=({\nu}_1,\ldots,{\nu}_s)$. We note that the first term is the number of non-crossing lattice paths from $(0,0),(1,0),\ldots,(s-1,0)$ to the points $(1,{\nu}_1),(2,{\nu}_{2}),\ldots,(s,{\nu}_s)$ and the the rest terms are from the one of non-crossing lattice paths from $(0,0),(1,0),\ldots, (k-1,0)$ to $(1,k),(2,k-1),\ldots, (k,1)$. These numbers are deduced from the Gessel-Viennot formula \cite{GV89} (see also \cite[p.23]{ACT}). 
 
 In addition, at $y=-1$, the value of $d_{\mathbf{k}}$ is the number of strict partitions inside the partition ${\nu}(\mathbf{k})$ and containing the partition ${\nu}(\boldsymbol{\beta})$. By Definition \ref{def:beta} and Lemma \ref{l:4.8}, the number of strict partitions containing ${\nu}(\boldsymbol{\beta})$ and inside ${\nu}(\mathbf{k})$ is the same as the number of partitions in ${\nu}(\boldsymbol{\beta})\backslash{\nu}(\mathbf{k})$, and this number is given by a disjoint union of  shorter weakly decreasing partitions, for instance, $\nu^{(1)},\nu^{(2)},\ldots$. Thus, the number of partitions in ${\nu}(\boldsymbol{\beta})\backslash{\nu}(\mathbf{k})$ is the product of the number of partitions inside $\nu^{(i)}$. Furthermore, the number of partitions inside each $\nu^{(i)}=(\nu^{(i)}_1\geq\cdots\geq\nu^{(i)}_r)$ can be calculated by the binomial determinant 
\[
p(\nu^{(i)}):= \left|\binom{\nu^{(i)}_{j}+r-i+1}{1+j-i}\right|_{1\leq i,j\leq r}
\]
(See \cite[(16)]{ACT}). In other words, the coefficient $d_{\mathbf{k}}$ can be obtained as the product of the values $p(\nu^{(i)})$. 
 \end{ex}

\subsubsection{The motivic Hirzebruch class for even orthogonal degeneracy loci}\label{sec4.4.2}

Let $V_D$ be a rank $2n+2$ vector bundle on $X$, and $\langle\cdot,\cdot\rangle_D$ a non-degenerate symmetric form on $V_D$. Let $\pi_D:OG(V_D)\rightarrow X$ be the even orthogonal Grassmannian bundle of isotropic rank $n+1$ subbundles of $V_D$. The Grassmannian bundle $OG(V_D)$ breaks up into two connected components $OG'(V_D)$ and $OG''(V_D)$ depending on the parity condition. That is, for a fixed maximal isotropic bundle $V_0$ of rank $n+1$, $OG'(V_D)$ (respectively $OG''(V_D)$) parametrizes rank $n+1$ subbundles $U$ such that $\mathrm{dim}(U\cap V_0)\equiv n+1$ (mod $2$) (respectively $\mathrm{dim}(U\cap V_0)\equiv n$ (mod $2$)).

It is known that the orthogonal Grassmannian $OG(n, 2n+1)$ which parametrizes $n$-dimensional subspaces of a $2n+1$-dimensional vector space equipped with a nondegenerate symmetric bilinear form, and $OG(n+1,2n+2)$ for $(n+1)$-dimensional subspaces in a $2n+2$-dimensional space, are isomorphic \cite[\S3.5]{IMN}. In light of this, one can interpret this fact in terms of degeneracy loci, which will later be used to apply our formula to symplectic orbit closures in \S \ref{sec5.2.2}.

Let $E\subset V_B$ be any maximal isotropic subbundle of $V_B$. Let $M:=E^\perp/E$ be the line bundle, and $G=V_B/E^\perp,$ where $E^\perp$ is obtained with respect to $\langle\cdot,\cdot\rangle_B$. There is an embedding $\psi:V_B=E\oplus M\oplus G\hookrightarrow V_D=E\oplus M\oplus N\oplus G$, where $N=V_D/M^\perp\cong M^\vee$ is a rank $1$ subbundle of $V_D$ such that $M^\perp$ is taken with respect to $\langle\cdot,\cdot\rangle_D$. In particular, the symmetric form on $V_B$ is the restriction of the form on $V_D$ to the summands $E\oplus M\oplus G=V_B$. 

Furthermore, we have a map $\eta:OG(V_B)\rightarrow OG'(V_D)$ (or respectively $\eta':OG(V_B)\rightarrow OG''(V_D)$) induced by the embedding $\psi$. If we take a maximal isotropic bundle $E\in OG(V_B)$, then $E\oplus M$ and $E\oplus N$ are the only maximal isotropic subbundles of $V_D$ one of which satisfies the parity condition for $OG'(V_D)$ and for $OG''(V_D)$. This implies that $\eta$ is isomorphism, and its inverse is the restriction to the rank $n$ isotropic subbundle $\mathcal{G}_1$ of $\mathcal{G}\in OG'(V_D)$.

  Then we have the following lemma for their Hirzebruch classes.
\begin{lemma}\label{eqn4.8}
Let $\widehat{\Omega}_\lambda^D$ be the resolution of the even orthogonal degeneracy loci $\Omega_\lambda^D$. After applying $\eta_*$,
the push-forward of Hirzebruch classes of $\widehat{\Omega}_\lambda^D$ and $\widehat{\Omega}_\lambda^B$ agree in $A_*(X)\otimes \mathbb{Q}[y]$, i.e.,
\[
(\pi_D\iota_D)_*T_y(\widehat{\Omega}_\lambda^D)=(\pi_B\iota_B)_*T_y(\widehat{\Omega}_\lambda^B).
\]
where $\iota_D$ is the inclusion $\widehat{\Omega}_\lambda^D\hookrightarrow OG(V_D)$.
\end{lemma}
\begin{proof}
Without loss of generality, we let $\mathcal{Z}$ is either $M$ or $N$. Let $\mathfrak{W}$ be the maximal isotropic vector bundle on $OG'(V_D)$. The image of $\widehat{\Omega}_\lambda^B\subset OG(V_B)$ under the map $\eta$ is a subvariety $\widehat{\Omega}_\lambda^D\subset OG'(V_D)$ corresponding to the same partition $\lambda$ as the closure of the locus  defined by the conditions
 \begin{equation}\label{eqn4.6}
 \begin{array}{l}
\mathrm{dim}(\mathcal{W}\cap F_{q_i})= i\;\text{for}\;i=1,\ldots,s,  \quad\text{and}\quad \mathrm{dim}(F_0\cap \mathcal{W})\equiv n+1\;\text{(mod $2$)}
  \end{array}
\end{equation}
such that $\eta^{-1}(\widehat{\Omega}_\lambda^D)=\widehat{\Omega}_\lambda^B$. Here $\mathcal{W}$ is the tautological subbundle on $OG'(V_D)$.
 Similarly, we may have the image $\widehat{\Omega}_\lambda^D\subset OG''(V_D)$ under $\eta$ with the condition $\mathrm{dim}(F_0\cap \mathcal{W})\equiv n\;\text{(mod $2$)}$.
In addition, we have the map $\phi_D:\widehat{\Omega}_\lambda^D\rightarrow \Omega_\lambda^D$, where 
$
\Omega_\lambda^D\subset X$ defined by the same condition as $\widehat{\Omega}_\lambda^D$. 
Since the case for $OG''(V_D)$ is almost identical to the one for $OG'(V_D)$, we will focus only on $OG'(V_D)$. All these maps can be expressed as the following commutative diagram: 
\[
\begin{tikzcd}[row sep=scriptsize, column sep=scriptsize]
&  \widehat{\Omega}_\lambda^B \arrow[dl,hook',"\eta"'] \arrow[rr,hook,"\iota_B"] \arrow[dd,"\phi_B",near start] & &  OG(V_B) \arrow[dl,"\eta"'] \arrow[dd,"\pi_B"] \\
 \widehat{\Omega}_\lambda^D  \arrow[rr, crossing over, hook,"\iota_D"',near end] \arrow[dd,"\phi_D"] & & OG'(V_D)  \\
& \Omega_\lambda^B \arrow[dl,hook'] \arrow[rr,hook] & & X  \arrow[dl,equal] \\
  \Omega_\lambda^D \arrow[rr,hook] & & X \arrow[from=uu, crossing over,"\pi_D",near start]\\
\end{tikzcd}.
\]
Then there is a map $\eta_*:\pi_B^*(A_*(X)[y])\rightarrow \pi_D^*(A_*(X)[y])$ that restricts to an isomorphism, satisfying $\eta_*T_y(\widehat{\Omega}_\lambda^B)=T_y(\widehat{\Omega}_\lambda^D)$. Consequently we obtain that
\begin{equation}
\begin{split}
(\pi_D\iota_D)_*T_y(\widehat{\Omega}_\lambda^D)&=(\pi_D)_*(\iota_D)_*\eta_*T_y(\widehat{\Omega}_\lambda^B)=(\pi_D)_*f_*(\iota_B)_*T_y(\widehat{\Omega}_\lambda^B)\\
&=(\pi_B\iota_B)_*T_y(\widehat{\Omega}_\lambda^B). \qedhere
\end{split}
\end{equation}
\end{proof}
Accordingly, the class of the resolution in the odd (maximal) orthogonal case is isomorphic to that in even orthogonal case. This implies that the expansion of the CSM classes of the even orthogonal degeneracy loci, in terms of the Schubert classes associated to partitions $\lambda^D$, coincides with the corresponding expansion for the same partitions $\lambda^D$ in the the odd (maximal) orthogonal case.

\section{Orthogonal and Symplectic Orbit Closures}\label{sec5}
In this section, as applications of the formulas exhibited in the previous sections, we compute the motivic Hirzebruch classes of the orthogonal and symplectic orbit closures. The guiding idea behind these computations is that the graphs $\mathcal{G}:=G(f)\subset \mathbb{C}^n\oplus(\mathbb{C}^n)^\vee$ of a (symmetric or skew-symmetric) linear map $f$ are isomorphic with respect to the relevant forms, and rank conditions defining $K$-orbit closures associated with vexillary involutions can be translated into incidence conditions $\mathrm{dim}(\mathcal{G}\cap \mathcal{F}_{\mu_i})\geq i$ with a fixed relevant isotropic flag $\mathcal{F}_\bullet$. In the following subsections, we formalize this translation and provide the precise definitions for these relations (Lemmas \ref{lem5.1} and \ref{lem5.6}).

We begin with the definitions related the orbit closures associated with vexillary involutions. Main references for definitions of the orthogonal orbit closures and their connection to the degeneracy loci are \cite{MP,Brendan}. 

Let $G=GL_n(\mathbb{C})$ be the general linear group of invertible $n\times n$ complex matrices for a positive integer $n$. Let $B\subseteq G$ be the Borel subgroup of invertible lower triangular matrices. We denote by $Fl_n=G/B$ the flag variety. 
Let $\theta=\theta^{-1}$ be a holomorphic involution of $G$ and $K_n\subset G$ be the fixed point subgroup of the involution $\theta$, i.e.,  $K=G^\theta=\{g\in G:\theta(g)=g\}$. It is a standard fact that there are finitely many $K_n-$orbits in $Fl_n$ \cite{T79,Kaz,Jose}.

Let $\theta$ be the involution defined by $\theta(g)=(g^t)^{-1}$, where $g^t$ is the transpose of $g$. Then $K_n$ is the orthogonal group ${O}_n\subset G$. Assume now that $n=2k$ is even. Let $1_k$ denote the $k\times k$ identity matrix and
\[
J=\begin{bmatrix}
   0 & 1_k \\
   -1_k& 0
\end{bmatrix}.
\]
If $\theta$ is defined by $\theta(g)=(-Jg^tJ)^{-1}$, then $K\subset G$ becomes the symplectic group ${Sp}_n$.

Additionally, ${O}_n$-orbits on $Fl_n$ are in bijection with the sets of involutions
\[
I_n:=\left\{w\in S_n\;|\;w=w^{-1}\right\}
\]
in the symmetric group $S_n$, and for even $n$, ${Sp}_n$-orbits correspond bijectively to the sets of fixed-point-free involutions 
\[
I_n^{FPF}:=\left\{z\in I_n;|\;z(i)\neq i\;\text{for all}\; 1\leq i\leq n\right\}.
\]
For simplicity, we use the notations $I_n^O$ to represent $I_n$ and $I_n^{Sp}$ to denote $I_n^{FPF}$.

Let $\mathbb{C}^n$ be a complex vector space of dimension $n$ and $(\mathbb{C}^n)^\vee$ its dual space of $\mathbb{C}$-linear maps $\mathbb{C}^n\rightarrow \mathbb{C}$.
On the direct sum space $\mathbb{C}^n\oplus(\mathbb{C}^n)^\vee$, we can define the symplectic and symmetric bilinear forms $\langle\;,\;\rangle^-$ and $\langle\;,\;\rangle^+$ by
$
\langle v_1\oplus u_1,v_2\oplus u_2\rangle^\mp= u_1(v_2)\mp u_2(v_1)
$ for $v_1, v_2\in \mathbb{C}^n$ and $u_1,u_2\in (\mathbb{C}^n)^\vee.$ 
Using these forms, we can define the Lagrangian Grassmannain ${LG}_{2n}$ as the closed subvariety
\[
{LG}_{2n}:=\left\{U\in Gr(n,\mathbb{C}^n\oplus(\mathbb{C}^n)^\vee)\;|\; \langle U,U\rangle^-\equiv 0\right\}
\]
of $Gr(n,2n)=Gr(n,\mathbb{C}^n\oplus(\mathbb{C}^{n})^\vee)$, and the orthogonal Grassmannian $OG_{2n}$ by
\[
{OG}_{2n}:=\left\{U\in Gr(n,\mathbb{C}^n\oplus(\mathbb{C}^n)^\vee)\;|\; \langle U,U\rangle^+\equiv 0\right\}.
\]
In fact, $OG_{2n}$ has two components, and here we only consider $OG_{2n}$ to be the one containing $\mathbb{C}^n$. Two isotropic subspace $G_1$ and $G_2$ belong to the same connected component if and only if the dimension of their intersection satisfies $\mathrm{dim}(G_1\cap G_2)\equiv n$ (mod $2$).

Let $f:\mathbb{C}^n\rightarrow(\mathbb{C}^n)^\vee$ be a linear map. Let $G(f):=\{(v, f(v))\;|\; v\in \mathbb{C}^n\}$ be the graph of $f$. Then we have the following fact: the map $f$ is symmetric, meaning $f(v)(w)=f(w)(v)$ for $v,w\in \mathbb{C}^n$, if and only if $G(f)\in LG_{2n}$; and $f$ is skew-symmetric, that is, $f(v)(w)=-f(w)(v)$, if and only if $G(f)\in OG_{2n}$. We note that since $\mathrm{dim}(G(f)\cap \mathbb{C}^n)=0$, the graph $G(f)$ actually lies in $OG_{2n}$. Given an invertible symmetric (or skew-symmetric) map $f$, the subgroup of $G$ preserving the symmetric form $(u,v)\mapsto f(v)(u)$ (or skew-symmetric form) is the orthogonal group ${O}_n$ (or symplectic group $Sp_n$).

Let $A_{[i][j]}$ denote the upper-left $i\times j$ submatrix of $A$. Let $E_\bullet=Bg\in Fl_n$ be a flag represented by some $g\in G$ such that $E_i$ is generated by the first $i$ rows of $g$. Let $K_n\in \{O_n,Sp_n\}$. For $z\in I_n^{K_n}$ and a fixed invertible symmetric (respectively, skew-symmetric) map $f$ corresponding to $O_n$ (respectively $Sp_n$), the $K_n$-orbit closure $X_z^{K_n}$ is defined by the closure of 
\begin{equation}\label{eqn5.1}
\left\{E_\bullet\in Fl_n\;\bigm|\;\mathrm{rank}(E_j\xrightarrow{f} E_i^\vee)=\mathrm{rank}(z_{[i][j]})\;\text{for}\; i,j\in[n]\right\}.
\end{equation}
Here $z$ is identified with its permutation matrix, and $[m]=\{1,\ldots,m\}$.

Let $y\in I_n^{K_n}$. The {\it orthogonal Rothe diagram} of $y$ is 
\[
D^{O_n}(y)=\left\{(i,y(j))\;|\; y(i)>y(j)\leq i<j\; \text{for}\;(i,j)\in[n]\times[n]\right\},
\]
and the {\it symplectic Rothe diagram} is
\[
D^{Sp_n}(y)=\left\{(i,y(j))\;|\; y(i)>y(j)< i<j\; \text{for}\;(i,j)\in[n]\times[n]\right\}.
\]

The rank conditions in \eqref{eqn5.1} are redundant, allowing the condition to be expressed in terms of fewer, essential rank conditions. Let $Ess(D)$ be the {\it essential set} of a set $D\subset \mathbb{N}\times\mathbb{N}$ given by 
$
\{(i,j)\in D\;|\;(i,j+1)\notin D\;\text{and}\;(i+1,j)\notin D\}.
$ The $K_n$-orbit closure associated to $z\in I_n^{K_n}$ is the closure of
\[
\left\{E_\bullet\in Fl_n\;\bigm|\;\mathrm{rank}(E_j\xrightarrow{f} E_i^\vee)=\mathrm{rank}(z_{[i][j]})\;\text{for}\; (i,j)\in Ess(D^{K_n}(z))\right\}
\]
for some fixed invertible symplectic or symmetric $f$ depending on $K_n=O_n$ or $Sp_n$.

\subsection{Vexillary involutions}\label{sec:vex}
For $K_n=O_n$ or $Sp_n$, an involution $z\in I_n^{K_n}$ is {\it vexillary} if and only if $\mathrm{Ess}(D^{K_n}(z))$ forms a chain under the partial order $\preceq$ on $\mathbb{Z}\times\mathbb{Z}$, where $(a,b)\preceq (i,j)$ if and only if $i\leq a$ and $b\leq j$. Especially, if $K_n=Sp_n$, we call such an involution {\it Sp-vexillary}. It is worth noting that a vexillary involution is not necessarily Sp-vexillary.

We say that $z\in I_n^{O_n}$ is {\it vexillary} if $\mathrm{Ess}(D^{O_n}(z))=\{(i_1,j_1)\preceq (i_2,j_2)\preceq \cdots\preceq (i_s,j_s)\}$ (\cite[Lemma 28]{Brendan}). Let $X=Fl_n$. We consider a trivial bundle $\mathcal{V}_O=\mathbb{C}^n\oplus(\mathbb{C}^n)^\vee$ over $X$ equipped with a symplectic from on it. 

We fix an isotropic partial flag 
\begin{equation}\label{eqn5.2}
\mathcal{E}_{j_1}\oplus\mathcal{E}_{i_1}^\perp\;\scalebox{1.2}{$\subset$}\;\cdots\;\scalebox{1.2}{$\subset$}\;\mathcal{E}_{j_s}\oplus\mathcal{E}_{i_s}^\perp\;\scalebox{1.2}{$\subset$}\; \mathcal{V}^{O_n}
\end{equation}
with respect to the symplectic form and induced by the tautological bundle $\mathcal{E}_i$ over $X$, $\mathrm{rk}(\mathcal{E}_i)=i$ for $i=1,\ldots, s$. 
Given a fixed invertible symmetric map $\mathsf{f}:\mathbb{C}^n\rightarrow(\mathbb{C}^n)^\vee$ of vector bundles over $X$, we define $\mathcal{G}$ to be the trivial bundle $G(\mathsf{f})$ over $X$. Here $\mathcal{G}$ is a maximal isotropic subbundle of $\mathcal{V}^{O_n}$. The orthogonal orbit closure $X_z^{O_n}\subseteq Fl_n$ associated to the vexillary $z$ is given by 
\[
X_z^{O_n}=\{x\in X\;|\;\mathrm{dim}( \mathcal{G}\;\scalebox{1.1}{$\cap$}\; (\mathcal{E}_{j_t}\oplus\mathcal{E}_{i_t}^\perp))|_x\geq k_t\;\text{for }t=1,\ldots,s\}
\]
where $k_t=j_t-\mathrm{rk}(z_{[i_t][j_t]})$. 
We define a partition $\lambda^{O_n}(z)$ by $\lambda_{k}=i_t-j_t+1+k_t-k$ for $k=1,\ldots, k_s$ such that $k_{t-1}<k\leq k_t$, $k_0=0$.

We define $z\in I_n^{Sp_n}$ to be {\it Sp-vexillary} if it is vexillary and satisfies $\mathrm{Ess}(D^{Sp_n}(z))=\{(i_1,j_1)\preceq (i_2,j_2)\preceq \cdots\preceq (i_s,j_s)\}$ (\cite[Definition 21]{Brendan}). Let $\mathcal{V}^{Sp_n}$ be a trivial bundle $\mathbb{C}^n\oplus(\mathbb{C}^n)^\vee$ over $X=Fl_n$, equipped with a symmetric bilinear form. 

Let $\mathcal{E}_i$ be the tautological bundle of rank $i$ over $X$ for $1\leq i\leq s$. We fix an isotropic partial flag 
\begin{equation}\label{eqn5.3}
\mathcal{E}_{j_1}\oplus\mathcal{E}_{i_1}^\perp\;\scalebox{1.2}{$\subset$}\;\cdots\;\scalebox{1.2}{$\subset$}\;\mathcal{E}_{j_s}\oplus\mathcal{E}_{i_s}^\perp\;\scalebox{1.2}{$\subset$}\; \mathcal{V}^{Sp_n}
\end{equation}
with respect to the skew-symmetric form. 
Let $\mathcal{G}$ be the trivial bundle $G(\mathsf{f})$ over $X$ for the fixed invertible skew-symmetric map $\mathsf{f}:\mathbb{C}^n\rightarrow(\mathbb{C}^n)^\vee$ of vector bundles over $X$. 
The symplectic orbit closure $X_z^{Sp_n}\subseteq Fl_n$ associated to the Sp-vexillary $z$ is given by the closure of
\[
(X_z^{Sp_n})^\circ=\left\{x\in X\;|\;\mathrm{dim}( \mathcal{G}\cap (\mathcal{E}_{j_t}\oplus\mathcal{E}_{i_t}^\perp))|_x= k_t\;\text{for }t=1,\ldots,s\right\}
\]
where $k_t=j_t-\mathrm{rk}(z_{[i_t][j_t]})$. 
A partition $\lambda^{Sp_n}(z)$ is given by $\lambda^{Sp_n}_{k}=i_t-j_t+k_t-k$ for $k=1,\ldots, k_s$ such that $k_{t-1}<k\leq k_t$, $k_0=0$. 

\subsection{Motivic class of orthogonal orbit closures}
In this section, we present our main results on the motivic Hirzebruch class of orthogonal and symplectic orbit closures associated to vexillary involutions.

\subsubsection{Orthogonal orbit closures}

The key idea is to translate the problem of orthogonal orbit closures into the setting of Lagrangian Grassmannian degeneracy loci, enabling the application of our main theorems on the degeneracy loci.

\begin{lemma}\label{lem5.1}
Let $z\in I_n^{O_n}$ be a vexillary involution, and let $\lambda=\lambda^{O_n}(z)$. The orthogonal orbit closure $X_z^{O_n}$ can be identified with the Lagrangian Grassmannian degeneracy locus $\Omega_\lambda^C$.
\end{lemma}
\begin{proof}
Given the partial flag \eqref{eqn5.2} given by $\mathcal{F}_{\mu_{k_t}}=\mathcal{E}_{j_t}\oplus\mathcal{E}_{i_t}^\perp$, ${\mu_{k_t}}=i_t-j_t+1$ for $t=1,\ldots,s$, we get a refined isotropic flag 
\[
\mathcal{F}_{\mu_1'}\;\scalebox{1.2}{$\subset$}\;\mathcal{F}_{\mu_2'}\;\scalebox{1.2}{$\subset$}\; \cdots\;\scalebox{1.2}{$\subset$}\;\mathcal{F}_{\mu_{k_s}'}
\]
of subbundles of $\mathcal{V}^{O_n}$ by inserting 
$
\mathcal{F}_{\mu_{k_t-1}}=\mathcal{E}_{j_t-1}\oplus\mathcal{E}_{i_t}^\perp,
$
 $\mu_{k_t-1}=\mu_{k_t} +1$ between $\mathcal{F}_{\mu_{k_{t-1}}}$ and $\mathcal{F}_{\mu_{k_t}}$ whenever $k_t-k_{t-1}>1$. 
Since the condition $\mathrm{dim}(\mathcal{G}\cap \mathcal{F}_{\mu_{k_t}})\geq k_t$ for all $t=1,\ldots,s$ implies $\mathrm{dim}(\mathcal{G}\cap \mathcal{F}_{\mu_{k_t}-1})\geq k_t-1$, 
$X_z^{O_n}$ can be expressed as
\begin{equation}\label{eqn5.4}
\Omega_\lambda^C=\left\{x\in X\;|\;\mathrm{dim}(\mathcal{G}\cap \mathcal{F}_{\mu_i})|_x\geq i, i=1,\ldots, \ell\right\}.
\qedhere\end{equation}
\end{proof}

We denote by $\widetilde{X}_z^{O_n}$ the resolution of singularity for $X_z^{O_n}=\Omega_\lambda^C$ with a projection $\widetilde{X}_z^{O_n}\rightarrow X_z^{O_n}$ and the inclusion $\iota:\widetilde{X}_z^{O_n}\hookrightarrow X_s$ as in Theorem \ref{mainC}. 

It is known from \cite[\S 5.2]{Brendan} that $c(\mathcal{V}^{O_n}-\mathcal{F}_{\mu_t'}-\mathcal{G})=c(\mathcal{V}^{O_n})/(c(\mathcal{G})c(\mathcal{E}_{j_t}\oplus\mathcal{E}_{i_t}^\perp))=c(\mathcal{E}_{i_t}^\vee-\mathcal{E}_{j_t})$ as $c(\mathcal{G})$ becomes $1$.
For the notations described from the previous section, we let $c(k_t)=c(\mathcal{E}_{i_t}^\vee-\mathcal{E}_{j_t}),$ for $1\leq t\leq s$, and 
$
c(k)=c(\mathcal{E}_{i_t}^\vee-\mathcal{E}_{j_t-1})
$
where $i$ is the minimal so that $k_i\geq k$. 
Similarly, we let $T_y(k_t)=T_y(R_t\otimes(\mathcal{E}_{i_t}^\vee-\mathcal{E}_{j_t}))$ for $1\leq t\leq s$ and $T_y(k)=T_y(R_t\otimes(\mathcal{E}_{i_t}^\vee-\mathcal{E}_{j_t-1}))$ for $k_{i}\geq k>k_{i-1}$. Set $\ell=k_s$.  
Then we have the following theorem on the Motivic Hirzebruch classes.
\begin{thm}
Let $z\in I_n^{O_n}$ be vexillary with the partition $\lambda=\lambda^{O_n}(z)$. Let $X_z^{O_n}$ be the corresponding orthogonal orbit closure. 
The class $(\pi\iota)_*T_y(\widetilde{X}_z^{O_n})$ is given by 
\[
\prod_{j\leq i-1}\dfrac{T_y(R_i+R_{j})}{T_y(R_i-R_j)}\mathrm{Pf}_\lambda(d(1),\ldots,d(\ell))\scalebox{1.2}{$\cap$}\; T_y(Fl_n),
\]
where $d(i)=c(i)/T_y(i)$ if $\ell$ is even,
 and if $\ell$ is odd, then $d(i)=c(i)/T_y(i)$ for $1\leq i\leq \ell$, and $d(i)=1$ for $i=\ell+1$.
\end{thm}
\begin{proof}
The statement follows from Corollary \ref{cor3.6}, as all $\mu_i$ in \eqref{eqn5.4} are positive. 
\end{proof}

 Then, as a corollary, we have the formulas for the CSM classes of the orthogonal orbit closures associated to the vexillary $z$. We recall $c_t(i):=\sum_{j\geq0}c(i)_jt^j$. 
\begin{cor}
Let $z\in I_n^{O_n}$ be vexillary, with $\lambda=\lambda^{O_n}(z)$. The CSM class $(\pi\iota)_*c_{SM}(\widetilde{X}_z^{O_n})$ is 
\[
\prod_{j\leq i-1}\dfrac{1+R_i+R_{j}}{1+R_i-R_j}\prod_{i=1}^\ell\dfrac{(1+R_i)^{(1-\mu_i)}}{c_{\frac{1}{1+R_i}}(i)}\mathrm{Pf}_\lambda(c(1),\ldots,c(\ell))\scalebox{1.2}{$\cap$}\; c_{SM}(Fl_n).
\]
\end{cor}
\begin{proof}
By taking $y = -1$, we can directly apply \eqref{ACTlem4.2} to obtain the result.
\end{proof}

As before, the above class can be decomposed into strata of the orbit closures, but no stratum of codimension one in $X_z^{O_n}$ exists.
Here are some examples of the CSM class of orthogonal orbit closures associated to vexillary involutions. 
\begin{ex}
Let $X=Fl_2$ and consider the orthogonal orbit closure $X_z^{O_2}$ associated with $z=21\in I_2^{O_2}$. Then Ess$(D^{O_2}(z))=\{(1,1)\}$ and $\lambda^{O_2}(z)=(1)$. With $c(1)=c(\mathcal{E}_1^\vee-\mathcal{E}_1)$ and $c(2)=1$, we have
\begin{align*}
c_{SM}({X}_z^{O_2})&=\left((1-c_1(1)+c_2(1)+\cdots)c_1(1)\right)\scalebox{1.2}{$\cap$}\; c_{SM}(Fl_2)\\
&=[X_z^{O_2}]\scalebox{1.2}{$\cap$}\; c_{SM}(Fl_2).
\end{align*}
\end{ex}

\begin{ex}
Let $X=Fl_4$. We consider $X_z^{O_4}$ associated with $z=3412=(1\;3)(2\;4)$ in $I_4^{O_4}$. The corresponding partition is $\lambda^{O_4}(z)=(2,1)$, so that
we have
\begin{align*}
c_{SM}({X}_z^{O_4})&=\dfrac{1+R_2+R_1}{1+R_2-R_1}\dfrac{(1+R_1)^{-1}}{c_{\frac{1}{1+R_1}}(1)}\dfrac{1}{c_{\frac{1}{1+R_2}}(2)}(c_2(1)c_1(2)-2c_3(1))\scalebox{1.2}{$\cap$}\; c_{SM}(Fl_4)\\
&=\left([X_z^{O_4}]-3[X_{y}^{O_4}]\right)\scalebox{1.2}{$\cap$}\; c_{SM}(Fl_4).
\end{align*}
where $c(1)=c(\mathcal{E}_3^\vee-\mathcal{E}_2), c(2)=c(\mathcal{E}_2^\vee-\mathcal{E}_2)$, $y=4321=(14)(23)$ with $\lambda^{O_4}(y)=(3,1)$.

\end{ex}

\subsubsection{Symplectic orbit closures}\label{sec5.2.2}

Similarly, we convert the symplectic orbit closures into the even orthogonal degeneracy loci as in the following lemma. 

\begin{lemma} \label{lem5.6}
Let $z\in I_n^{Sp_n}$ be a vexillary involution, and the partition $\lambda=\lambda^{Sp_n}(z)$. The symplectic orbit closure $X_z^{Sp_n}$ can be realized as the even  orthogonal degeneracy locus $\Omega_{\lambda}^D$.
\end{lemma}
\begin{proof}
Let us fix an isotropic partial flag $\mathcal{F}_{\mu_{k_1}}\subset \mathcal{F}_{\mu_{k_2}}\subset\cdots\subset \mathcal{F}_{\mu_{k_s}}$ of subbundles on $X$, as defined in \eqref{eqn5.3}, where $\mathcal{F}_{\mu_{k_t}}=\mathcal{E}_{j_t}\oplus\mathcal{E}_{i_t}^\perp$ and $\mu_{k_t}=i_t-j_t$ for $t=1,\ldots,s$.

The remainder of the proof follows the same argument as in the orthogonal case. We refine the partial flag by inserting intermediate isotropic subbundles 
$
\mathcal{F}_{\mu_{k_t-1}}=\mathcal{E}_{j_t-1}\oplus\mathcal{E}_{i_t}^\perp
$ whenever $k_t-k_{t-1}>1$, yielding a refined flag 
\begin{equation}\label{eqn5.5}
\mathcal{F}_{\mu_1’}\;\scalebox{1.2}{$\subset$}\;\mathcal{F}_{\mu_2’}\;\scalebox{1.2}{$\subset$}\; \cdots\;\scalebox{1.2}{$\subset$}\;\mathcal{F}_{\mu_{k_s}’}.
\end{equation} 
The condition $\mathrm{dim}(\mathcal{G}\cap \mathcal{F}_{\mu_{k_t}})= k_t$ ensures that $\mathrm{dim}(\mathcal{G}\cap \mathcal{F}_{\mu_{k_t}-1})= k_t-1$ holds for any point $x\in X$, and the same implication holds for the closures of the corresponding degeneracy loci. Hence, the desired conclusion follows.
\end{proof}

Due to Lemma \ref{lem5.6}, we can regard $X_z^{Sp_n}$ as $\Omega_\lambda^D$ and consider the natural projection map $\phi_{Sp_n}:\widehat{X}_z^{Sp_n}=\widehat{\Omega}_\lambda^D\rightarrow X_z^{Sp_n}$. The fiber of $\phi_{Sp_n}$ over a point $x\in X=Fl_n$ is the closure of
\[
\left\{{\mathscr{G}}\in OG'(\mathcal{V}^{Sp_n})\;|\;\mathrm{dim}(\mathscr{G}\cap \mathcal{F}_{\mu_t'})= t\;\text{for }t=1,\ldots,\ell\right\}
\]
where $k_h=j_h-\mathrm{rk}(z_{[i_h][j_h]})$. Recall that $\mathscr{G}\in\widehat{X}_z^{Sp_n}$ and $\mathbb{C}^n$ are in the same component of $OG(\mathcal{V}^{Sp_n})$.

Let $c(k_t)=c(\mathcal{E}_{i_t}^\vee-\mathcal{E}_{j_t})$ for each $1\leq t\leq s$. For a given $k$, we let
$
c(k)=c(\mathcal{E}_{i_t}^\vee-\mathcal{E}_{j_t-1})
$
where $t$ is the smallest index such that $k_t\geq k$. 
Similarly, we denote $T_y(k_t)=T_y(R_t\otimes(\mathcal{E}_{i_t}^\vee-\mathcal{E}_{j_t}))$ for $1\leq t\leq s$ and $T_y(k)=T_y(R_t\otimes(\mathcal{E}_{i_t}^\vee-\mathcal{E}_{j_t-1}))$ for $k_{t}\geq k>k_{t-1}$. Letting $\ell:=k_s$, we state the following theorem on the Motivic Hirzebruch classes.

\begin{thm}\label{thm5.7}
Let $z\in I_n^{Sp_n}$ be vexillary with the partition $\lambda=\lambda^{Sp_n}(z)$. Let $X_z^{Sp_n}$ be the corresponding symplectic orbit closure. 
Then the class $(\pi\iota)_*T_y(\widehat{X}_z^{Sp_n})$ is isomorphic to
\[
\dfrac{1}{2^\ell}\prod_{j\leq i-1}\dfrac{T_y(R_i+R_{j})}{T_y(R_i-R_j)}\prod_{i=1}^\ell \left(T_y(R_i)\right)^2\mathrm{Pf}_\lambda(d(1),\ldots,d(\ell))\scalebox{1.2}{$\cap$}\; T_y(Fl_n),
\]
where $d(i)=c(i)/T_y(i)$ if $\ell$ is even,
 and if $\ell$ is odd, then $d(i)=c(i)/T_y(i)$ for $1\leq i\leq \ell$, and $d(\ell+1)=1$.
\end{thm}

\begin{proof}
As discussed above from Lemma \ref{lem5.6}, $\widehat{X}_z^{Sp_n}$ can be identified with $\widehat{\Omega}_\lambda^D$ in the even orthogonal case. We pullback this locus under the map $\eta$ in \S\ref{sec4.4.2} to get the odd orthogonal degeneracy locus
$
\widehat{\Omega}_\lambda^B=\eta^{-1}\widehat{\Omega}_\lambda^D
$
given by
\[
\{\mathcal{S}\in OG(\mathcal{V}_B)\;|\;\mathrm{dim}(\mathcal{S}\cap \mathcal{F}_{\mu_t'})\geq t\;\text{for}\; t=1,\ldots,\ell\}
\]
where $\mathcal{V}_B=\eta^{-1}\mathcal{V}^{Sp_n}$ is a vector bundle of rank $2n-1$, equipped with the symmetric form and $\mathcal{S}=\eta^{-1}(\mathcal{G})$ of rank $n-1$ on $X$. 

We then apply Corollary \ref{cor4.5}, using the fact that $\mu_i'>0$ for all $i$, to get the motivic Hirzebruch class 
\[
\dfrac{1}{2^\ell}\prod_{j\leq i-1}\dfrac{T_y(R_i+R_{j})}{T_y(R_i-R_j)}\prod_{i=1}^\ell T_y(R_i)\mathrm{Pf}_\lambda(\widetilde{d}(1),\ldots,\widetilde{d}(\ell))\scalebox{1.2}{$\cap$}\; T_y(X),
\]
where
$
\widetilde{d}(i)=c(\mathcal{V}_B-\mathcal{F}_{\mu_i'}-\mathcal{S}-\mathcal{M})/T_y(R_i\otimes(\mathcal{V}_B-\mathcal{F}_{\mu_i'}-\mathcal{S}-\mathcal{M}))
$ for $i=1,\ldots,\ell$. Here $\mathcal{M}$ is the line bundle such that $\mathcal{M}\cong \mathcal{F}^\perp/\mathcal{F}$ for some maximal isotropic subbundle $\mathcal{F}$ of $\mathcal{V}_B$. 

Furthermore, we know that $c_1(\mathcal{M})=0$ so that $c(\mathcal{V}_B-\mathcal{F}_{\mu_i'}-\mathcal{S}-\mathcal{M})=c(\mathcal{V}_B-\mathcal{F}_{\mu_i'}-\mathcal{S})$ and $1/T_y(R_i\otimes(\mathcal{V}_B-\mathcal{F}_{\mu_i'}-\mathcal{S}-\mathcal{M}))=T_y(R_i\otimes\mathcal{M})/T_y(R_i\otimes(\mathcal{V}_B-\mathcal{F}_{\mu_i'}-\mathcal{S}))$.

Altogether, using Lemma \ref{eqn4.8}, we arrive at the statement. 
\end{proof}

 Then, as a corollary, we have the formulas for the CSM classes of the orthogonal orbit closures associated to the vexillary $z$. We recall $c_t(i):=\sum_{j\geq0}c(i)_jt^j$. 
\begin{cor}
Let $z\in I_n^{Sp_n}$ be vexillary. The CSM class $(\pi\iota)_*c_{SM}(\widehat{X}_z^{Sp_n})$ is isomorphic to
\[
\dfrac{1}{2^\ell}\prod_{j\leq i-1}\dfrac{1+R_i+R_{j}}{1+R_i-R_j}\prod_{i=1}^\ell\dfrac{(1+R_i)^{(2-\mu_i')}}{c_{\frac{1}{1+R_i}}(i)}\mathrm{Pf}_\lambda(c(1),\ldots,c(\ell))\scalebox{1.2}{$\cap$}\; c_{SM}(Fl_n)
\]
\end{cor}

\begin{proof}
Specializing $y$ to $-1$ from Theorem \ref{thm5.7} and using \eqref{ACTlem4.2} leads to the desired result.
\end{proof}

We end this section with the following examples of the CSM classes of the symplectic orbit closures associated to vexillary involutions.
\begin{ex}
Let $X=Fl_5$. We take ${X}_z^{Sp_5}$ associated with $z=43215=(1\;4)(2\;3)\in I_5^{Sp_5}$. Then Ess$(D^{Sp_5}(z))=\{(3,1)\}$ and $\lambda^{Sp_5}(z)=(2)$. For $c(1)=c(\mathcal{E}_3^\vee-\mathcal{E}_1)$ and $c(2)=1$, one has
\begin{align*}
c_{SM}({X}_z^{Sp_5})&=\dfrac{1}{2}\left((1- c_1(1) +c_2(1) + c_1(1)R_1+\cdots)c_2(1)\right)\scalebox{1.2}{$\cap$}\; c_{SM}(Fl_5)\\
&=\left([X_z^{Sp_5}]-2[X_{y_1}^{Sp_5}]-2[X_{y_2}^{Sp_5}]+6[X_{y_3}^{Sp_5}]\right)\scalebox{1.2}{$\cap$}\; c_{SM}(Fl_5),
\end{align*}
where $y_1=(1\;4)$, $y_2=(1\;4)(2\;5)$ and $y_3=(1\;5)(2\;4)$.
\end{ex}

\begin{ex}
Let $X=Fl_6$ and ${X}_z^{Sp_6}$ be the symplectic orbit closure associated with $z=543216=(1\;5)(2\;4)\in I_6^{Sp_6}$. Then the corresponding partition becomes
$\lambda^{Sp_6}(z)=(3,1)$, and 
we have
\begin{align*}
c_{SM}({X}_z^{Sp_6})&=\dfrac{1}{2^2}\left(\dfrac{1+R_2+R_1}{1+R_2-R_1}\dfrac{(1+R_1)^{-1}}{c_{\frac{1}{1+R_1}}(1)}\dfrac{(1+R_2)^{-3}}{c_{\frac{1}{1+R_2}}(2)}(c_3(1)c_1(2)-2c_4(1))\right)\scalebox{1.2}{$\cap$}\; c_{SM}(Fl_6)\\
&=\left([X_z^{Sp_6}]-3[X_{y_1}^{Sp_6}]-3[X_{y_2}^{Sp_6}]+11[X_{y_3}^{Sp_6}]\right)\scalebox{1.2}{$\cap$}\; c_{SM}(Fl_6),
\end{align*}
where $c(1)=c(\mathcal{E}_4^\vee-\mathcal{E}_1), c(2)=c(\mathcal{E}_3^\vee-\mathcal{E}_2)$, $y_1=(1\;5)(2\;6), y_2=(1\;6)(2\;4)$ and $y_3=(1\;6)(2\;5)$.
\end{ex}

\begin{ack}
The author wishes to express our gratitude to David Anderson for his insightful comments, which are incorporated in this work. We also thank anonymous referees for carefully reviewing our manuscript and for the many helpful suggestions that greatly improved the exposition. The author is supported by the Institute for Basic Science (IBS-R032-D1) and was partially supported by AMS-Simons Travel Grant. 
\end{ack}


\begin{bibdiv}
\begin{biblist}

\bib{AM16}{article}{
      author={Aluffi, Paolo},
      author={Mihalcea, Leonardo~C.},
       title={Chern-{S}chwartz-{M}ac{P}herson classes for {S}chubert cells in
  flag manifolds},
        date={2016},
     journal={Compos. Math.},
      volume={152},
      number={12},
       pages={2603\ndash 2625},
}

\bib{AMSS22}{article}{
      author={Aluffi, Paolo},
      author={Mihalcea, Leonardo~C.},
      author={Sch\"urmann, J\"org},
      author={Su, Changjian},
       title={From motivic chern classes of schubert cells to their hirzebruch
  and csm class},
        date={2022},
     journal={preprint, arxiv:2212.12509},
}

\bib{AMSS23}{article}{
      author={Aluffi, Paolo},
      author={Mihalcea, Leonardo~C.},
      author={Sch\"urmann, J\"org},
      author={Su, Changjian},
       title={Shadows of characteristic cycles, {V}erma modules, and positivity
  of {C}hern-{S}chwartz-{M}ac{P}herson classes of {S}chubert cells},
        date={2023},
     journal={Duke Math. J.},
      volume={172},
      number={17},
       pages={3257\ndash 3320},
}

\bib{AMSS24}{article}{
      author={Aluffi, Paolo},
      author={Mihalcea, Leonardo~C.},
      author={Sch\"urmann, J\"org},
      author={Su, Changjian},
       title={Motivic {C}hern classes of {S}chubert cells, {H}ecke algebras,
  and applications to {C}asselman's problem},
        date={2024},
     journal={Ann. Sci. \'Ec. Norm. Sup\'er. (4)},
      volume={57},
      number={1},
       pages={87\ndash 141},
}

\bib{ACT}{article}{
      author={Anderson, Dave},
      author={Chen, Linda},
      author={Tarasca, Nicola},
       title={Motivic classes of degeneracy loci and pointed {B}rill-{N}oether
  varieties},
        date={2022},
     journal={J. Lond. Math. Soc. (2)},
      volume={105},
      number={3},
       pages={1787\ndash 1822},
}

\bib{A}{article}{
      author={Anderson, David},
       title={{$K$}-theoretic {C}hern class formulas for vexillary degeneracy
  loci},
        date={2019},
     journal={Adv. Math.},
      volume={350},
       pages={440\ndash 485},
}

\bib{AF}{article}{
      author={Anderson, David},
      author={Fulton, William},
       title={Chern class formulas for classical-type degeneracy loci},
        date={2018},
     journal={Compos. Math.},
      volume={154},
      number={8},
       pages={1746\ndash 1774},
}

\bib{Kaz}{article}{
      author={Aomoto, Kazuhiko},
       title={On some double coset decompositions of complex semisimple {L}ie
  groups},
        date={1966},
     journal={J. Math. Soc. Japan},
      volume={18},
       pages={1\ndash 44},
}

\bib{BSY}{article}{
      author={Brasselet, Jean-Paul},
      author={Sch\"urmann, J\"org},
      author={Yokura, Shoji},
       title={Hirzebruch classes and motivic {C}hern classes for singular
  spaces},
        date={2010},
     journal={J. Topol. Anal.},
      volume={2},
      number={1},
       pages={1\ndash 55},
}

\bib{Ful}{book}{
      author={Fulton, William},
       title={Intersection theory},
     edition={Second},
      series={Ergebnisse der Mathematik und ihrer Grenzgebiete. 3. Folge. A
  Series of Modern Surveys in Mathematics [Results in Mathematics and Related
  Areas. 3rd Series. A Series of Modern Surveys in Mathematics]},
   publisher={Springer-Verlag, Berlin},
        date={1998},
      volume={2},
}

\bib{GV89}{article}{
      author={Gessel, I.~M.},
      author={Viennot, X.~G.},
       title={Determinants, paths, and plane partitions},
        date={1989},
     journal={Preprint},
       pages={23},
}

\bib{GJL}{article}{
      author={Graham, William},
      author={Jeon, Minyoung},
      author={Larson, Scott},
       title={Irreducible characteristic cycles for orbit closures of a
  symmetric subgroup},
     journal={In preparation.},
}

\bib{HIMN}{article}{
      author={Hudson, Thomas},
      author={Ikeda, Takeshi},
      author={Matsumura, Tomoo},
      author={Naruse, Hiroshi},
       title={Degeneracy loci classes in {$K$}-theory---determinantal and
  {P}faffian formula},
        date={2017},
     journal={Adv. Math.},
      volume={320},
       pages={115\ndash 156},
}

\bib{HM}{article}{
      author={Hudson, Thomas},
      author={Matsumura, Tomoo},
       title={Vexillary degeneracy loci classes in {$K$}-theory and algebraic
  cobordism},
        date={2018},
     journal={European J. Combin.},
      volume={70},
       pages={190\ndash 201},
}

\bib{Huh16}{article}{
      author={Huh, June},
       title={Positivity of {C}hern classes of {S}chubert cells and varieties},
        date={2016},
     journal={J. Algebraic Geom.},
      volume={25},
      number={1},
       pages={177\ndash 199},
}

\bib{IMN}{article}{
      author={Ikeda, Takeshi},
      author={Mihalcea, Leonardo~C.},
      author={Naruse, Hiroshi},
       title={Factorial {$P$}- and {$Q$}-{S}chur functions represent
  equivariant quantum {S}chubert classes},
        date={2016},
     journal={Osaka J. Math.},
      volume={53},
      number={3},
       pages={591\ndash 619},
}

\bib{Jones}{article}{
      author={Jones, Benjamin~F.},
       title={Singular {C}hern classes of {S}chubert varieties via small
  resolution},
        date={2010},
     journal={Int. Math. Res. Not. IMRN},
      number={8},
       pages={1371\ndash 1416},
}

\bib{KL}{article}{
      author={Kempf, G.},
      author={Laksov, D.},
       title={The determinantal formula of {S}chubert calculus},
        date={1974},
     journal={Acta Math.},
      volume={132},
       pages={153\ndash 162},
}

\bib{MP}{article}{
      author={Marberg, Eric},
      author={Pawlowski, Brendan},
       title={{$K$}-theory formulas for orthogonal and symplectic orbit
  closures},
        date={2020},
     journal={Adv. Math.},
      volume={372},
       pages={107299, 43},
}

\bib{T79}{article}{
      author={Matsuki, Toshihiko},
       title={The orbits of affine symmetric spaces under the action of minimal
  parabolic subgroups},
        date={1979},
     journal={J. Math. Soc. Japan},
      volume={31},
      number={2},
       pages={331\ndash 357},
}

\bib{PP95}{article}{
      author={Parusi\'nski, Adam},
      author={Pragacz, Piotr},
       title={Chern-{S}chwartz-{M}ac{P}herson classes and the {E}uler
  characteristic of degeneracy loci and special divisors},
        date={1995},
     journal={J. Amer. Math. Soc.},
      volume={8},
      number={4},
       pages={793\ndash 817},
}

\bib{Brendan}{article}{
      author={Pawlowski, Brendan},
       title={Universal graph {S}chubert varieties},
        date={2021},
     journal={Transform. Groups},
      volume={26},
      number={4},
       pages={1417\ndash 1461},
}

\bib{PR22}{incollection}{
      author={Promtapan, Sutipoj},
      author={Rim\'anyi, Rich\'ard},
       title={Characteristic classes of symmetric and skew-symmetric degeneracy
  loci},
        date={2022},
   booktitle={Facets of algebraic geometry. {V}ol. {II}},
      series={London Math. Soc. Lecture Note Ser.},
      volume={473},
   publisher={Cambridge Univ. Press, Cambridge},
       pages={254\ndash 283},
}

\bib{RS90}{article}{
      author={Richardson, R.~W.},
      author={Springer, T.~A.},
       title={The {B}ruhat order on symmetric varieties},
        date={1990},
     journal={Geom. Dedicata},
      volume={35},
      number={1-3},
       pages={389\ndash 436},
}

\bib{Tam11}{article}{
      author={Tamvakis, Harry},
       title={Giambelli, {P}ieri, and tableau formulas via raising operators},
        date={2011},
     journal={J. Reine Angew. Math.},
      volume={652},
       pages={207\ndash 244},
}

\bib{Tanisaki}{incollection}{
      author={Tanisaki, Toshiyuki},
       title={Holonomic systems on a flag variety associated to
  {H}arish-{C}handra modules and representations of a {W}eyl group},
        date={1985},
   booktitle={Algebraic groups and related topics ({K}yoto/{N}agoya, 1983)},
      series={Adv. Stud. Pure Math.},
      volume={6},
   publisher={North-Holland, Amsterdam},
       pages={139\ndash 154},
}

\bib{Jose}{article}{
      author={Wolf, Joseph~A.},
       title={The action of a real semisimple group on a complex flag manifold.
  {I}. {O}rbit structure and holomorphic arc components},
        date={1969},
     journal={Bull. Amer. Math. Soc.},
      volume={75},
       pages={1121\ndash 1237},
}

\bib{WY}{article}{
      author={Wyser, B.},
      author={Yong, A.},
       title={Polynomials for symmetric orbit closures in the flag variety},
        date={2017},
     journal={Transform. Groups},
      volume={22},
      number={1},
       pages={267\ndash 290},
}

\bib{Wyser}{article}{
      author={Wyser, Benjamin~J.},
       title={{$K$}-orbit closures on {$G/B$} as universal degeneracy loci for
  flagged vector bundles with symmetric or skew-symmetric bilinear form},
        date={2013},
     journal={Transform. Groups},
      volume={18},
      number={2},
       pages={557\ndash 594},
}

\bib{Yok}{incollection}{
      author={Yokura, Shoji},
       title={A singular {R}iemann-{R}och for {H}irzebruch characteristics},
        date={1998},
   booktitle={Singularities {S}ymposium---\l ojasiewicz 70 ({K}rak\'ow, 1996;
  {W}arsaw, 1996)},
      series={Banach Center Publ.},
      volume={44},
   publisher={Polish Acad. Sci. Inst. Math., Warsaw},
       pages={257\ndash 268},
}

\bib{Zhang18}{article}{
      author={Zhang, Xiping},
       title={Chern classes and characteristic cycles of determinantal
  varieties},
        date={2018},
     journal={J. Algebra},
      volume={497},
       pages={55\ndash 91},
}

\end{biblist}
\end{bibdiv}
\end{document}